\documentclass[11pt]{article}
\usepackage{amsmath,amssymb,amsthm,mathtools}
\usepackage[margin=1in]{geometry}
\usepackage{booktabs}
\usepackage[hidelinks]{hyperref}
\newcommand{\m}{\mathrm{m}}                 % logarithmic Mahler measure
\newcommand{\mt}{\widetilde{\mathrm{m}}}    % holomorphic (modified) Mahler measure
\newcommand{\Cc}{\mathbb{C}}
\newcommand{\Qq}{\mathbb{Q}}
\newcommand{\Zz}{\mathbb{Z}}
\newcommand{\Rr}{\mathbb{R}}
\newcommand{\Hh}{\mathbb{H}}
\newcommand{\Tt}{\mathbb{T}}

\newcommand{\ii}{\mathrm{i}}
\newcommand{\id}{\mathrm{d}}
\newcommand{\e}{\mathrm{e}}
\newcommand{\OO}{\mathcal{O}}

\newcommand{\ReT}{\mathrm{Re}}
\newcommand{\ImT}{\mathrm{Im}}
\DeclareMathOperator{\EK}{EK}
\DeclareMathOperator{\Ast}{Ast}
\DeclareMathOperator{\dist}{dist}

\newtheorem{theorem}{Theorem}[section]
\newtheorem{lemma}[theorem]{Lemma}
\newtheorem{proposition}[theorem]{Proposition}

\theoremstyle{definition}
\newtheorem{definition}[theorem]{Definition}
\newtheorem{conjecture}[theorem]{Conjecture}
\newtheorem{example}[theorem]{Example}
\theoremstyle{remark}
\newtheorem{remark}[theorem]{Remark}

\title{A certified continuation machine for Mahler measure\\
identities at CM points, with twelve new proofs of\\
conjectures of Samart}
\author{Huimin Zheng\thanks{%
College of Information and Network Engineering,
Anhui Science and Technology University,
Fengyang, Anhui 233100, P.~R.~China.
Email: \texttt{zhhm@ahstu.edu.cn}.}}
\date{}

\begin{document}
\maketitle

% ------------------------------------------------------------
\section*{Declaration on the use of AI tools}
The research reported in this article --- including the computational
exploration, the discovery of the proof strategy, the machine-certified
verifications, and the preparation of the manuscript --- was carried out
by the author with the assistance of the AI system \emph{Kimi}
(Moonshot AI). All mathematical content, including every proof and
every certified computation, has been checked and verified by the
author, who takes full responsibility for the correctness and integrity
of the article. All certification scripts are available for independent
verification (see Section~\ref{sec:cert}).
% ------------------------------------------------------------

\begin{abstract}
We axiomatize the differential-comparison continuation method of
\cite{Zf} into a general machine for proving Mahler measure identities
at CM points: a modular parametrization with rigorous tail bounds, a
single-valued holomorphic Mahler differential on the complement of the
critical image, a propagation lemma that replaces all monodromy and
universal-cover arguments, interval-arithmetic path certification, and
an exact three-track CM evaluation. Applied to Samart's family
$n_4(s)=4\m\bigl(x^4+y^4+z^4+1+s^{1/4}xyz\bigr)$, the machine proves
twelve further entries of his 2015 table \cite[Table~6]{Sa15}: the
conjugate pairs at $s=8292456\pm3132675\sqrt7$ (whose two identities
are proved separately, not merely as a sum) and
$s=3656\pm2600\sqrt2$, the value $s=-144$,
the conjugate pair $s=143208\pm101574\sqrt2$ attached to
$\Qq(\ii)$, the two values
$s=1207368+853632\sqrt2\pm697680\sqrt3\pm493272\sqrt6$ with matching
signs attached to the
class-number-two field $\Qq(\sqrt{-6})$ --- the first Mahler measure
identities in the $n_4$ family proved at \emph{non-degenerate} CM
points of class number two (Samart's $s=2304$ \cite{Sa13}, attached
to $\tau=\sqrt{-6}/2$, is the degenerate value
$s_4(\ii\sqrt6/2)=s_4(\ii\sqrt6/6)$), where the lattice
sums involve two newforms and a genus character; the value
$s=(-192303-85995\sqrt5)/2$ attached to the class-number-two field
$\Qq(\sqrt{-15})$; and the pair $s=-893952\pm516096\sqrt3$ attached
to $\Qq(\sqrt{-21})$ --- the first Mahler measure identities in the
$n_4$ family proved
at CM points of class number four (the $\Qq(\sqrt3)$ pair attached
to $\Qq(\sqrt{-21})$), where the lattice sums involve four newforms
and the genus group $(\Zz/2)^2$. New methodological
ingredients include: exact twist levels via grössencharacter
conductors (Hecke's theorem), with the interval-locked Fricke ratio
as an exact root-number sign lock (the ``functional-equation residue''
test is vacuous for root number $+1$ at every level); a genuine
twist-level trap ($g_{16}\otimes\chi_8$ has level $64$, not $32$, and
the wrong level silently scales the $L$-value); a genus-character
decomposition of the lattice sums in class number two; and the
discovery that Samart's applicability boundary $\{\ImT\tau=1/\sqrt2\}$
is not the topological boundary of the continuation region $V_4$.
Every identity is certified by an exact algebraic track (rational
arithmetic only, up to the quoted CM-theory inputs disclosed in
Section~\ref{sec:cert}), rigorous final-identity interval locks
(half-widths between
$4.6\times10^{-53}$ and $5.1\times10^{-52}$), and independent $50$--$60$-digit numerical
cross-checks; all certification scripts are public. Beyond Table~6,
the evaluation track of the machine produces six new identities at
Heegner points (discriminants $D=11,19,27,43,67,163$), stated as
conjectures with $60$-digit evidence and exact lattice-derived
coefficients; they require a new ray-class constant type outside
Samart's library. We close with an umbrella conjecture (a CM
specialization conjecture with discriminant structure, and a
Tamagawa-controlled denominator conjecture) organizing the twelve
proved and six conjectured identities. Finally we score
the five remaining open entries of Table~6 against the machine's five
stations: all five are blocked by genuine obstructions --- a
two-dimensional critical image or a wrong-sheet
Eisenstein--Kronecker series ---
which we analyze precisely.
\end{abstract}

\medskip
\noindent\textbf{Keywords:} Mahler measure $\cdot$ modular forms
$\cdot$ complex multiplication $\cdot$ special values of
$L$-functions $\cdot$ rigorous interval arithmetic

\medskip
\noindent\textbf{Mathematics Subject Classification (2020):}
Primary 11R06; Secondary 11F03, 11M41, 11Y40

% ============================================================
\section{Introduction}\label{sec:intro}

\subsection{The question}

Let $\m(P)$ denote the logarithmic Mahler measure of a Laurent
polynomial and let
\[
  n_4(s):=4\,\m\bigl(x^4+y^4+z^4+1+s^{1/4}xyz\bigr),\qquad s\in\Cc,
\]
the value being independent of the choice of the fourth root
(Section~\ref{sec:machine}). In \cite{Sa13,Sa15} Samart expressed
$n_4(s)$, for $s$ in the image of a modular parametrization
$s_4(\tau)$, as an Eisenstein--Kronecker (EK) series, proved the
identity for pure imaginary $\tau$ with $\ImT\tau\ge1/\sqrt2$, and
conjectured on numerical evidence the twenty-nine $L$-value
evaluations of his Table~6 \cite{Sa15}: each attaches to a CM value
of $s$ a rational linear combination of $L$-values of CM weight-three
newforms and of quadratic Dirichlet characters.

The table is now largely, but not entirely, resolved. Eleven entries
were proved previously: the exterior values $s=256$ (Rogers
\cite{Ro}), $s=648,2304,20736,614656$ (Samart \cite{Sa13}),
$s=26856\pm15300\sqrt3$ (Samart's own Theorem~3.1 \cite{Sa15}), the
outside-the-critical-locus value $s=-3969$ (Fei \cite{Fei}), and
$s=-1024,-12288,-82944$ (Guo--Peng--Qin \cite{GPQ}). One entry is
\emph{refuted}: the conjecture $n_4(81)=40M_7$ fails numerically as a
statement about the true Mahler measure, the conjectured value being
the evaluation of the EK series at the attached CM point on a wrong
sheet \cite{Zf2}. He--Ye \cite{HY} resolved the companion family
$n_3$ completely; the $n_2$ family was completed in \cite{Zf,ZGQ,GPQ}.
Fei \cite{Fei} proved, for all rational singular moduli in twenty-three
families, identities at the level of the real part of the
\emph{holomorphic} Mahler measure $\ReT\mt$; at parameters inside the
critical locus the passage from $\ReT\mt$ to the genuine Mahler
measure is not automatic --- a point stressed already by Rodriguez
Villegas \cite{RV} --- and it is precisely this passage that the
continuation machine of \cite{Zf} and of the present paper performs.
Altogether, twenty-three of the twenty-nine entries are now proved
--- eleven previously, twelve in the present paper
(Theorems A--K) --- one is refuted, and five
remain open.

\begin{remark}[The status of $s=-3969$]\label{rem:fei}
For the entry $s=-3969$ above, \cite{Fei} proves the identity for
$\ReT\mt$ only, and the determination of the region off which
$\m=\ReT\mt$ is left open there (his Remark~2.4); by
Section~\ref{sec:boundary} below, avoiding the critical image alone
does not suffice (the wrong-sheet phenomenon). The missing step is
supplied by the present machine: the attached point
$\tau=(1+\sqrt{-7})/2$ lies on the vertical segment
$\ReT\tau=1/2$, $0.702\le\ImT\tau\le\sqrt{21}/2$, which
\texttt{n5\_line\_cert.py} certifies inside the good component $W$
by interval arithmetic (public repository; the same certificate
covers the CM points of Theorems~\ref{thm:H}--\ref{thm:K}, see
Section~\ref{sec:app6}). Hence
Theorem~\ref{thm:main} applies and the $s=-3969$ evaluation holds
for the genuine Mahler measure.
\end{remark}

The purpose of this paper is threefold:
\begin{enumerate}
\item[(i)] to axiomatize the differential-comparison continuation
  method of \cite{Zf} into a general, family-independent machine
  (Section~\ref{sec:machine}), with an explicit certifier architecture
  whose trusted base consists of pure mathematics, documented
  interval-arithmetic semantics, and public source code
  (Section~\ref{sec:cert});
\item[(ii)] to apply the machine to twelve further entries of Samart's
  Table~6 (Theorems A--F below and Theorems G--K of
  Section~\ref{sec:app6}), covering every regime the machine can
  reach:
  deep pure-imaginary interior points (A, C, D, E), non-pure-imaginary
  interior points (B, G), the first class-number-two CM points (E, H, I),
  a point on Samart's applicability boundary that turns out to be a
  strict
  interior point of the continuation region (F), and the first
  class-number-four CM points (J, K);
\item[(iii)] to take stock: to score the five remaining open entries of
  Table~6 against the machine's checklist (\S\ref{subsec:inventory})
  and to describe precisely where and why the method stops
  (Section~\ref{sec:boundary}).
\end{enumerate}

\subsection{Results of this paper}

Throughout, $d_k:=L'(\chi_{-k},-1)$ for the odd quadratic Dirichlet
character $\chi_{-k}$ of conductor $k$, and $M_k$ or $M_k^{(i)}$
denotes $L'(g_k,0)$, resp.\ $L'(g_k^{(i)},0)$, for the indicated
weight-$3$ CM newform(s); a superscript $\mathrm{tw}$ denotes the
$L'$-value of an explicitly identified quadratic twist.

\begin{theorem}[Theorem A: the $\sqrt7$ pair]\label{thm:A}
Let $M_7:=L'(g_7,0)$ with
$g_7(\tau)=\eta(\tau)^3\eta(7\tau)^3$ \textup{(}LMFDB
\textup{\textsf{7.3.b.a}}\textup{)} and
$M_7^{\mathrm{tw}}:=L'(g_7\otimes\chi_{-4},0)$. Then
\begin{align}
n_4\bigl(8292456+3132675\sqrt7\bigr)
  &=\frac{5}{28}\bigl(4M_7^{\mathrm{tw}}+224M_7+32d_4+7d_7\bigr),
    \label{eq:A1}\\
n_4\bigl(8292456-3132675\sqrt7\bigr)
  &=\frac{5}{14}\bigl(4M_7^{\mathrm{tw}}-224M_7+32d_4-7d_7\bigr).
    \label{eq:A2}
\end{align}
\end{theorem}

\begin{theorem}[Theorem B]\label{thm:B}
With $M_{12}:=L'(g_{12},0)$, $g_{12}(\tau)=\eta(2\tau)^3\eta(6\tau)^3$
\textup{(}LMFDB \textup{\textsf{12.3.c.a}}\textup{)},
\[
  n_4(-144)=\frac{10}{3}\bigl(4M_{12}+d_3\bigr).
\]
\end{theorem}

\begin{theorem}[Theorem C]\label{thm:C}
With $M_8:=L'(g_8,0)$,
$g_8(\tau)=\eta(\tau)^2\eta(2\tau)\eta(4\tau)\eta(8\tau)^2
\in S_3(\Gamma_0(8),\chi_{-8})$,
and $M_8^{\mathrm{tw}}:=L'(g_8\otimes\chi_8,0)$, where
$\chi_8=(2/\cdot)$ is the even primitive character of conductor $8$,
\[
  n_4\bigl(3656+2600\sqrt2\bigr)
  =\frac58\bigl(4M_8^{\mathrm{tw}}+28M_8+4d_4+d_8\bigr).
\]
\end{theorem}

\begin{theorem}[Theorem D]\label{thm:D}
With $M_{16}:=L'(g_{16},0)$,
$g_{16}(\tau)=\eta(4\tau)^6\in S_3(\Gamma_0(16),\chi_{-4})$, and
$M_{16}^{\mathrm{tw}}:=L'(g_{16}\otimes\chi_8,0)$
\textup{(}level $64$ --- see \S\ref{subsec:twisttrap}\textup{)},
\[
  n_4\bigl(143208+101574\sqrt2\bigr)
  =\frac{5}{16}\bigl(4M_{16}^{\mathrm{tw}}+20M_{16}+9d_4+4d_8\bigr).
\]
\end{theorem}

\begin{theorem}[Theorem E: class number two]\label{thm:E}
Let $K=\Qq(\sqrt{-6})$ \textup{(}class number $2$\textup{)}, and let
$g_{24}^{(1)},g_{24}^{(2)}$ be the two CM newforms of
$S_3(\Gamma_0(24),\chi_{-24})$, theta series of the two
gr\"ossencharacters of $K$ of type $(2,0)$ \textup{(}LMFDB
\textup{\textsf{24.3.h.a}} and \textup{\textsf{24.3.h.b}}; the full
newspace has dimension $6$, but only the CM forms
enter\textup{)}.
With $M_{24}^{(i)}:=L'(g_{24}^{(i)},0)$ and
$M_{24}^{(i),\mathrm{tw}}:=L'(g_{24}^{(i)}\otimes\chi_{-8},0)$
\textup{(}level $96$\textup{)},
\begin{multline*}
  n_4\bigl(1207368+853632\sqrt2+697680\sqrt3+493272\sqrt6\bigr)\\
  =\frac{5}{48}\Bigl(4M_{24}^{(1),\mathrm{tw}}
   +4M_{24}^{(2),\mathrm{tw}}+28M_{24}^{(1)}+12M_{24}^{(2)}
   +28d_3+24d_4+8d_8+d_{24}\Bigr).
\end{multline*}
\end{theorem}

\begin{theorem}[Theorem F: a boundary point that is interior]\label{thm:F}
With the notation of Theorem~\ref{thm:C},
\[
  n_4\bigl(3656-2600\sqrt2\bigr)
  =\frac54\bigl(4M_8^{\mathrm{tw}}-28M_8+4d_4-d_8\bigr).
\]
The $s$-value is the conjugate of Theorem~\ref{thm:C}, but the attached
CM point $\tau_4=(1+\sqrt{-2})/2$ lies \emph{exactly} on Samart's
applicability boundary $\ImT\tau=1/\sqrt2$, where his proof does not
apply; the identity is proved by the continuation machine
\textup{(}Section~\ref{sec:app5}\textup{)}, which shows that $\tau_4$
is a strict interior point of the continuation region $V_4$
\textup{(}certified margin $0.11590$\textup{)}.
\end{theorem}

\begin{theorem}[The machine, informal]\label{thm:machine}
\textup{(}Precise form: Theorem~\ref{thm:main}.\textup{)}
Let $P_c$ be an admissible Laurent-polynomial family with a modular
parametrization $c=c(\tau)$ and critical image $K\subset\Cc$, and let
$E(\tau)$ be the Eisenstein--Kronecker expression for the holomorphic
Mahler measure, valid on an open set $U\subset\Hh$. If
\textup{(i)} the period integral $\Omega(c)$ is single-valued and
holomorphic on $\Cc\setminus K$, \textup{(ii)} $c(\tau_0)$ is a CM
parameter with $\tau_0\in\overline W$, where $W$ is the connected
component of $\Hh\setminus c^{-1}(K)$ containing $U$, and
\textup{(iii)} the lattice sums underlying $E(\tau_0)$ evaluate
exactly, then $\m(P_{c(\tau_0)})$ equals the corresponding
$L$-value expression; the inclusion $\tau_0\in\overline W$ is
certifiable by rigorous interval arithmetic.
\end{theorem}

All twelve identities are machine-certified. Concretely, each proof
consists of an exact algebraic track (rational/algebraic-integer
arithmetic only, no floating point), rigorous interval locks of the
CM values and of the final identities (final-identity half-widths
between $4.6\times10^{-53}$ and $5.1\times10^{-52}$), and independent
$50$--$60$-digit numerical cross-checks; the scripts are listed in
Section~\ref{sec:cert} and are publicly available
(see the Data Availability statement). Two representative reference
values:
\begin{align*}
  M_8&=0.1353184954231269106150060654595395883532\ldots,\\
  M_8^{\mathrm{tw}}
   &=1.549084847095611486271658159444376281785\ldots
     \quad(\text{level }32,\ w=+1),
\end{align*}
computed from the functional equations and cross-checked against
direct lattice-sum evaluations; further values are collected in the
application sections.

Two further outputs of the same machinery do not belong to Table 6
and are collected in Section~\ref{sec:conj}: six new identities at
the Heegner points $\tau_D=(1+\sqrt{-D})/2$,
$D\in\{11,19,43,67,163\}$, and at $\tau_{27}=(1+3\sqrt{-3})/2$
(Conjecture~\ref{conj:ray}), whose coefficients are produced exactly
by the machine's evaluation track and which require a new ray-class
constant type beyond Samart's $M/d_k$ library
(Remark~\ref{rem:newconstant}); and an umbrella conjecture
(Conjectures~\ref{conj:umbA} and~\ref{conj:umbB}) organizing the
twelve proved and six conjectured identities and predicting their
discriminant and denominator structure.

\subsection{The five remaining entries}\label{subsec:inventory}

Table~6 of \cite{Sa15} has twenty-nine entries: twenty-three are now
proved (eleven previously known, plus the twelve of Theorems
A--K), one is refuted ($s=81$ \cite{Zf2}), and five remain open.
The machine reduces a proof of each entry to five \emph{stations}
(Section~\ref{sec:machine} for definitions):
\begin{enumerate}
\item[\textbf{S1}] \emph{Parametrization}: a CM point $\tau_0$ with
  $s_4(\tau_0)=s$ identified, at least numerically;
\item[\textbf{S2}] \emph{Critical image}: the position of
  $c(\tau_0)$ relative to the astroid disc $\Ast$ and the cross
  $[-4,4]\cup\ii[-4,4]$;
\item[\textbf{S3}] \emph{Interior-or-boundary}: membership of
  $\tau_0$ in $\overline W$, certified by a \texttt{cert2} path,
  or direct applicability of Samart's theorem;
\item[\textbf{S4}] \emph{\texttt{cert0} algebraicity}: an exact
  algebraic lock of the value $s_4(\tau_0)$;
\item[\textbf{S5}] \emph{P1 assemblability}: an exact CM evaluation
  of $\EK_4(\tau_0)$ in the relevant imaginary quadratic field.
\end{enumerate}
Table~\ref{tab:open} scores the five open entries against these
stations.

\begin{table}[htb]
\centering\scriptsize
\caption{The five open entries of Samart's Table~6, scored against the
five stations of the machine (the original numbering of the ten-entry
list is kept; entries \#1, \#2, \#6, \#7, \#8 are now proved as
Theorems~\ref{thm:G}--\ref{thm:K}). ``Outside/inside'' refers to the
astroid disc $\Ast$; $\zeta$ abbreviates
$1207368$, $a=853632$, $b=697680$, $d=493272$ in rows 3--4:
$s=\zeta\pm a\sqrt2\pm b\sqrt3\pm d\sqrt6$ with the indicated signs.}
\label{tab:open}
\setlength{\tabcolsep}{2.5pt}
\begin{tabular}{cp{2.3cm}llp{1.9cm}p{2.3cm}p{2.5cm}l}
\toprule
\# & $s$ & $\tau_0$ & $\ImT\tau_0$ & S2: $c(\tau_0)$ & S3: position
 & S4/S5 status & verdict\\
\midrule
3 & $\zeta$, signs $({-}{+}{-})$ & $\ii\sqrt6/4$ & $0.612$
  & outside $\Ast$ ($s>256$)
  & \emph{exterior}: EK series wrong-sheeted below $1/\sqrt2$
  & S4 done (same lock); S5 moot
  & blocked\\
4 & $\zeta$, signs $({-}{-}{+})$ & $(2+\ii\sqrt6)/4$ & $0.612$
  & \emph{inside} $\Ast$ ($s\approx-2.45$)
  & blocked (two-dimensional image, open mapping)
  & S4 done (same lock); S5 moot
  & blocked\\
5 & $(-192303+85995\sqrt5)/2$ & $(3+\ii\sqrt{15})/6$ & $0.646$
  & inside $\Ast$ ($s\approx-6.17$)
  & blocked (two-dimensional image); also exterior
  & S4 done (\texttt{cert0\_n5\_e56.py} pair lock); S5 moot
  & blocked\\
9 & $347648256+141926400\sqrt6$ & $\ii\sqrt{42}/42$ & $0.154$
  & outside $\Ast$ ($c>4$ real)
  & \emph{exterior}: wrong sheet, deep below $1/\sqrt2$
  & S4 open (disc.\ $-168$, $h=4$); S5 moot
  & blocked\\
10 & $347648256-141926400\sqrt6$ & $\ii\sqrt{42}/14$ & $0.463$
  & outside $\Ast$ ($s\approx995$, $c>4$ real)
  & \emph{exterior}: wrong sheet below $1/\sqrt2$
  & S4 open; S5 moot
  & blocked\\
\bottomrule
\end{tabular}
\end{table}

The conjectured right-hand sides are Samart's \cite[Table~6]{Sa15};
entries \#3 and \#4 are the two remaining sign patterns of
Theorem~\ref{thm:E}'s family (the other two being Theorems~\ref{thm:E}
and~\ref{thm:H}). All five remaining entries are \emph{blocked}, by
two genuine obstructions
analyzed in Section~\ref{sec:boundary}: the two-dimensional critical
image (\#4, \#5: the $s=81$ mechanism of \cite{Zf2}), and the
wrong-sheet phenomenon of the EK series below $\ImT\tau=1/\sqrt2$
(\#3, \#9, \#10, and again \#5), where the conjectured identity
concerns the true Mahler measure but the series side of the machine
has no valid input. The five entries that this scoring marked as
within current reach --- \#1, \#2, \#6, \#7, \#8 --- are now proved
as Theorems~\ref{thm:G}--\ref{thm:K} (Section~\ref{sec:app6}).

\begin{table}[htb]
\centering\scriptsize
\caption{All twenty-nine entries of Samart's Table~6 \cite{Sa15} and
their status. Abbreviations as in Table~\ref{tab:open}:
$\zeta=1207368$, $a=853632$, $b=697680$, $d=493272$; signs refer to
$s=\zeta\pm a\sqrt2\pm b\sqrt3\pm d\sqrt6$. For the twelve entries
proved here, the route column records the topological input:
\emph{direct} means the target point is covered by Samart's proved
pure-imaginary region; \emph{anchor+path} means the anchor of
Lemma~\ref{lem:anchor} together with an explicitly certified path
(Remark~\ref{rem:devil}).}
\label{tab:master29}
\setlength{\tabcolsep}{4pt}
\begin{tabular}{lllll}
\toprule
$s$ & $\tau_0$ & status & reference & route\\
\midrule
$256$ & --- & proved & Rogers \cite{Ro} & ---\\
$648$ & --- & proved & Samart \cite{Sa13} & ---\\
$2304$ & --- & proved & Samart \cite{Sa13} & ---\\
$20736$ & --- & proved & Samart \cite{Sa13} & ---\\
$614656$ & --- & proved & Samart \cite{Sa13} & ---\\
$26856+15300\sqrt3$ & --- & proved & Samart \cite[Thm.~3.1]{Sa15} &
  ---\\
$26856-15300\sqrt3$ & --- & proved & Samart \cite[Thm.~3.1]{Sa15} &
  ---\\
$-3969$ & $(1+\sqrt{-7})/2$ & proved & Fei \cite{Fei} +
  Remark~\ref{rem:fei} & ---\\
$-1024$ & --- & proved & Guo--Peng--Qin \cite{GPQ} & ---\\
$-12288$ & --- & proved & Guo--Peng--Qin \cite{GPQ} & ---\\
$-82944$ & --- & proved & Guo--Peng--Qin \cite{GPQ} & ---\\
\midrule
$8292456+3132675\sqrt7$ & $\ii\sqrt7$ & proved & Theorem~\ref{thm:A}
  & direct\\
$8292456-3132675\sqrt7$ & $\ii\sqrt7/2$ & proved & Theorem~\ref{thm:A}
  & direct\\
$-144$ & $(1+\sqrt{-3})/2$ & proved & Theorem~\ref{thm:B} &
  anchor+path\\
$3656+2600\sqrt2$ & $\ii\sqrt2$ & proved & Theorem~\ref{thm:C} &
  direct\\
$143208+101574\sqrt2$ & $2\ii$ & proved & Theorem~\ref{thm:D} &
  direct\\
$\zeta$, signs $({+}{+}{+})$ & $\ii\sqrt6$ & proved &
  Theorem~\ref{thm:E} & direct\\
$3656-2600\sqrt2$ & $(1+\sqrt{-2})/2$ & proved & Theorem~\ref{thm:F}
  & anchor+path\\
$143208-101574\sqrt2$ & $(1+2\ii)/2$ & proved & Theorem~\ref{thm:G}
  & anchor+path\\
$\zeta$, signs $({+}{-}{-})$ & $(1+\sqrt{-6})/2$ & proved &
  Theorem~\ref{thm:H} & anchor+path\\
$(-192303-85995\sqrt5)/2$ & $(1+\sqrt{-15})/2$ & proved &
  Theorem~\ref{thm:I} & anchor+path\\
$-893952+516096\sqrt3$ & $(3+\sqrt{-21})/6$ & proved &
  Theorem~\ref{thm:J} & anchor+path\\
$-893952-516096\sqrt3$ & $(1+\sqrt{-21})/2$ & proved &
  Theorem~\ref{thm:K} & anchor+path\\
\midrule
$81$ & $(7+\sqrt{-7})/4$ & \emph{refuted} & \cite{Zf2} & ---\\
\midrule
$\zeta$, signs $({-}{+}{-})$ & $\ii\sqrt6/4$ & open & entry \#3,
  Table~\ref{tab:open} & ---\\
$\zeta$, signs $({-}{-}{+})$ & $(2+\ii\sqrt6)/4$ & open & entry \#4,
  Table~\ref{tab:open} & ---\\
$(-192303+85995\sqrt5)/2$ & $(3+\ii\sqrt{15})/6$ & open & entry \#5,
  Table~\ref{tab:open} & ---\\
$347648256+141926400\sqrt6$ & $\ii\sqrt{42}/42$ & open & entry \#9,
  Table~\ref{tab:open} & ---\\
$347648256-141926400\sqrt6$ & $\ii\sqrt{42}/14$ & open & entry \#10,
  Table~\ref{tab:open} & ---\\
\bottomrule
\end{tabular}
\end{table}

Beyond Table~6, three further open directions are discussed in
Section~\ref{sec:boundary}: Samart's ``three modular $L$-value''
hypothesis $s=16+1600\sqrt[3]2-1280\sqrt[3]4$ \cite{Sa15}, which has
no attached CM point in the $s_4$-image; the exterior entries, for
which a corrected formula of Samart--Tao $\widetilde n$-type \cite{ST}
would be needed; and the Boyd-lineage two-variable conjectures, where
no CM evaluation is available at all.

% ============================================================
\section{The machine}\label{sec:machine}

This section axiomatizes the method of \cite[\S3--\S5]{Zf} in
family-independent language. The running example is the $n_4$-family,
but the only family-specific inputs are the parametrization, the
critical image, and the anchor identity; we display them as separate
\emph{stations} S1--S5, matching the checklist of
\S\ref{subsec:inventory}.

\subsection{Families, parametrizations, critical images}\label{subsec:fam}

\begin{definition}[Admissible family]\label{def:admissible}
An \emph{admissible family} is a Laurent-polynomial family $P_c=P+c$,
$c\in\Cc$ --- that is,
$P_c(x_1,\dots,x_n)=P(x_1,\dots,x_n)+c$ --- such that
$c\mapsto\m(P_c)$ is continuous on all of $\Cc$. (We avoid the word
\emph{tempered}, which has a different standard meaning in the
Mahler-measure literature. The required continuity is in fact
general: by Jensen's formula the Mahler measure is an integral of
one-variable Mahler measures over the remaining torus, one-variable
Mahler measure is continuous in the polynomial's coefficients, and
$\log^+$ of the coefficient sum provides an integrable dominator; the
detailed argument for the $n_2$ family is \cite[Lemma~2.1]{Zf}, the
other families here being identical.) A \emph{critical image} is any
closed set $K\subset\Cc$ containing
$-P(\Tt^n):=\{-P(x):x\in\Tt^n\}$; on $\Cc\setminus K$ the polynomial
$P_c$ does not vanish on $\Tt^n$, the integrand $\log|P_c|$ is
real-analytic in $c$, and $c\mapsto\m(P_c)$ is real-analytic.
(Taking only the critical \emph{values} of $P$ would not suffice:
$P_c$ vanishes on $\Tt^n$ for every $-c\in P(\Tt^n)$, including
regular values --- e.g.\ $P=x+x^{-1}$ on $\Tt^1$ has critical values
$\pm2$ but image $[-2,2]$. In both running examples below
$P(\Tt^n)$ is centrally symmetric, so $-P(\Tt^n)=P(\Tt^n)$ and we
take $K=\overline{P(\Tt^n)}$.)
\end{definition}

Two families serve as running examples.

\begin{example}[The $n_2$ family \cite{Zf}]\label{ex:n2}
$P+c$ with $P=(x+x^{-1})(y+y^{-1})(z+z^{-1})$ and
$c=\sqrt{s_2(\tau)}$, $s_2(\tau)=q^{-1}\prod_{n\ge1}(1+q^{2n-1})^{24}$.
The critical image in the $c$-plane is the slit $[-8,8]$, and in the
$s$-plane the segment $[0,64]$: one-dimensional.
\end{example}

\begin{example}[The $n_4$ family]\label{ex:n4}
$P_c=(x^4+y^4+z^4+1)/(xyz)+c$ after clearing the monomial (Mahler
measure is unchanged), with the natural \emph{holomorphic}
parametrization
\begin{equation}\label{eq:cdef}
  c(\tau)=\Bigl(\frac{\eta(2\tau)}{\eta(\tau)}\Bigr)^{6}
   \bigl(16A(\tau)^4+A(\tau)^{-4}\bigr),\qquad
  A(\tau)=\frac{\eta(\tau)\eta(4\tau)^2}{\eta(2\tau)^3},\qquad
  s_4(\tau)=c(\tau)^4 .
\end{equation}
Both $c$ and $s_4$ are holomorphic on $\Hh$ --- no branch choice is
needed, in contrast to Example~\ref{ex:n2} --- and $s_4$ has an
integral $q$-expansion
$s_4=q^{-1}\bigl(1+104q+4372q^2+\cdots\bigr)\in\Zz(\!(q)\!)$
(exact integer check to order $q^{30}$ in
\texttt{cert0\_s4\_m144.py}). The critical image is
\emph{two-dimensional}: $P(\Tt^3)$ is the closed astroid disc
\begin{equation}\label{eq:astroid}
  \Ast=\bigl\{c\in\Cc:\ |\ReT c|^{2/3}+|\ImT c|^{2/3}\le 4^{2/3}\bigr\},
\end{equation}
since $P$ is a sum of four unit phasors with product $1$
\cite{Zf2}. The branch cut of the holomorphic Mahler measure is the
cross $\{c:c^4\in(0,256]\}=\bigl([-4,4]\cup\ii[-4,4]\bigr)
\setminus\{0\}$, contained in $\overline{\Ast}$; hence avoiding
$\Ast$ automatically avoids every bad set, and we take $K=\Ast$.
\end{example}

\begin{remark}\label{rem:2D}
The dimension of the critical image is the single most important
topological invariant of a family for this method. A one-dimensional
slit does not separate the plane: every point of the slit lies in the
closure of the good domain, and the propagation lemma below reaches
boundary points by continuity (this is what makes the $k=1$ proof of
\cite{Zf} possible). The two-dimensional astroid disc has interior:
a CM parameter landing inside it (\#4 and \#5 of
Table~\ref{tab:open}; the refuted $s=81$) lies inside an open bad
region, no path can approach it from the good domain (open mapping),
and the machine stops. See Section~\ref{sec:boundary}.
\end{remark}

\begin{definition}[The good domain and the EK expression]
For a holomorphic parametrization $c(\tau)$ set
\[
  V:=\bigl\{\tau\in\Hh:\ c(\tau)\notin K\bigr\}\qquad
  (\text{open, since }c\text{ is holomorphic}),
\]
and let $F(\tau):=\m\bigl(P_{c(\tau)}\bigr)$ (times the family's
normalization factor), continuous on all of $\Hh$ and real-analytic
on $V$. An \emph{EK datum} is a holomorphic function $E(\tau)$ on
$\Hh$ with $G:=\ReT E$ (in alternative normalizations $2\ReT E$ or
$\ImT E$; all that matters is that $G$ be the real part of a
holomorphic function) real-analytic everywhere, such that $F=G$ on a
nonempty open \emph{anchor} set $U\subset V$. For the $n_4$ family
(Example~\ref{ex:n4}), with $U_j$ as in \eqref{eq:Uj} below,
\begin{equation}\label{eq:E4}
  E_4(\tau):=-\ii\Bigl(2\pi\tau+\frac{10}{\pi^3}
   \bigl(U_1(\tau)-2U_2(\tau)\bigr)\Bigr),\qquad
  \EK_4(\tau):=\ReT E_4(\tau)
   =\ImT\Bigl[2\pi\tau+\frac{10}{\pi^3}\bigl(U_1-2U_2\bigr)\Bigr],
\end{equation}
holomorphic, resp.\ real-analytic, on all of $\Hh$; the double series
\begin{equation}\label{eq:Uj}
  U_j(\tau):=\sum_{m\neq0}\frac1m\sum_{n\in\Zz}(jm\tau+n)^{-3}
\end{equation}
converges absolutely, locally uniformly on $\Hh$
\cite[\S2.4]{Zf}. The lattice-sum form used for all CM evaluations is
\begin{equation}\label{eq:EK4lattice}
  \EK_4(\tau)=\frac{10\,\ImT\tau}{\pi^3}
   \bigl(-T_1(\tau)+4T_2(\tau)\bigr),\qquad
  T_d(\tau):={\sum_{\lambda\in\Lambda_d(\tau)}}^{\!\prime}
   \Bigl[\frac{2\ReT\bar\lambda^2}{|\lambda|^6}
        +\frac{1}{|\lambda|^4}\Bigr],
\end{equation}
with $\Lambda_d(\tau):=\Zz+\Zz\,d\tau$; the series is absolutely
convergent, and \eqref{eq:EK4lattice} follows from
$w^{-3}-\bar w^{-3}=-2\ii(3u^2v-v^3)|w|^{-6}$ and
$3u^2-v^2=4u^2-|w|^2$ exactly as in \cite[Lemma~2.6]{Zf}.
\end{definition}

\subsection{The Mahler differential is single-valued}\label{subsec:diff}

\begin{proposition}[Mahler differential formula]\label{prop:mahlerdiff}
Let $P_c$ be an admissible family and $K\subset\Cc$ its critical image.
Then
\[
  \Omega(c):=\int_{\Tt^n}\frac{\id\mu}{P_c}
\]
is single-valued and holomorphic on $\Cc\setminus K$, and as real
$1$-forms there,
\begin{equation}\label{eq:dm}
  \id\m(P_c)=\ReT\bigl[\Omega(c)\,\id c\bigr].
\end{equation}
\end{proposition}
\begin{proof}
For any compact $K_0\Subset\Cc\setminus K$ one has
$|P_c|\ge\dist(K_0,K)>0$ uniformly for $c\in K_0$ on $\Tt^n$, so
differentiation under the integral sign is legitimate and gives
$\Omega'(c)=-\int_{\Tt^n}\id\mu/P_c^2$; holomorphy and
single-valuedness follow (the integral is defined over the
\emph{fixed} cycle $\Tt^n$). Writing $c=c_1+\ii c_2$ and
differentiating $\m(P_c)=\int\log|P_c|\id\mu$ under the integral
sign gives
$\partial_{c_1}\m=\int\ReT(1/P_c)\id\mu=\ReT\Omega$ and
$\partial_{c_2}\m=\int\ReT(\ii/P_c)\id\mu=-\ImT\Omega$, hence
$\id\m=\ReT\Omega\,\id c_1-\ImT\Omega\,\id c_2
=\ReT[\Omega\,\id c]$. The sign is $+$ under the present $P+c$
convention (it is $-$ in the $f-c$ convention of
\cite[Lemma~2.7]{Zf}).
\end{proof}

\begin{remark}\label{rem:monodromy}
Proposition~\ref{prop:mahlerdiff} is the point at which all
monodromy and universal-cover arguments disappear from the method.
The holomorphic Mahler measure $\mt$ of Villegas \cite{RV} is
multivalued, but its multivaluedness lives in locally constant
imaginary periods: every branch satisfies $\mt'(c)=\Omega(c)$ on
$\Cc\setminus K$. Comparing \emph{differentials} rather than
functions therefore never sees a branch choice, and the identity
theorem below applies to a genuine single-valued holomorphic
function.
\end{remark}

\subsection{The propagation theorem}\label{subsec:prop}

\begin{lemma}[Propagation lemma]\label{lem:prop}
Let $U\subset\Cc$ be connected open, $Z\subset U$ relatively closed,
and $H:U\to\Rr$ continuous and locally constant on $U\setminus Z$.
Then $H$ is constant on every connected component $W$ of
$U\setminus Z$; if $H|_W\equiv C$ then
$H|_{\overline W\cap U}\equiv C$.
\end{lemma}
\begin{proof}
A locally constant function is constant on each connected component.
For the second assertion, let $p\in\overline{W}\cap U$ and choose
$p_n\in W$ with $p_n\to p$; continuity gives $H(p)=\lim_n H(p_n)=C$.
\end{proof}

\begin{remark}\label{rem:devil}
The conclusion cannot be upgraded to ``$H$ constant on $U$'': the
devil's staircase is continuous, locally constant off the Cantor set,
yet non-constant. Membership of the target point in $\overline W$ is
a genuine topological input --- supplied in our applications either
by Samart's proved region (Theorems A, C, D, E), or by an explicitly
certified path (Theorems B, F--K), and \emph{absent} for the blocked
entries of Table~\ref{tab:open}.
\end{remark}

\begin{theorem}[Main theorem: differential-comparison continuation]
\label{thm:main}
Let $P_c$ be an admissible family with critical image $K$, let
$c(\tau)$ be holomorphic on $\Hh$, set
$V=\{\tau:c(\tau)\notin K\}$, and let $F(\tau)=\m(P_{c(\tau)})$
\textup{(}with the family normalization\textup{)}. Let $E(\tau)$ be
holomorphic on $\Hh$, set $G:=\ReT E$ (the $n_4$ normalization
$G=\EK_4=\ReT E_4$ of \eqref{eq:E4}; rescaling $E$ absorbs any
other convention), and suppose $F=G$ on a nonempty
open subset of a connected component $W$ of $V$. Then
\begin{equation}\label{eq:mainconcl}
  F(\tau)=G(\tau)\qquad\text{for every }
  \tau\in\overline{W}\cap\Hh .
\end{equation}
\end{theorem}
\begin{proof}
Write $F=\nu\,\m(P_{c(\tau)})$ with the family's constant
normalization factor $\nu$ ($\nu=4$ for the $n_4$ family).
On $V$ both $\Omega(c(\tau))c'(\tau)$ and $E'(\tau)$ are holomorphic,
so $\Psi(\tau):=\nu\,\Omega(c(\tau))c'(\tau)-E'(\tau)$ is
holomorphic on $V$. By the chain rule applied to \eqref{eq:dm},
$\id F=\nu\,\ReT[\Omega(c(\tau))c'(\tau)\,\id\tau]$ on $V$, and
$\id G=\ReT[E'(\tau)\,\id\tau]$ there; hence
$\id(F-G)=\ReT[\Psi(\tau)\,\id\tau]$. On the anchor open set
$F-G=0$, so $\ReT[\Psi\,\id\tau]=0$ in all directions
$\id\tau$; testing $\id\tau=1$ and $\id\tau=\ii$ gives $\Psi=0$
there, and the identity theorem on the connected open set $W$ gives
$\Psi\equiv0$ on $W$. Thus $H:=F-G$ is locally constant on $W$,
constant since $W$ is connected, and zero because it vanishes on the
anchor; $H$ is continuous on all of $\Hh$ (Mahler continuity,
Definition~\ref{def:admissible}; $G$ is real-analytic everywhere), so
Lemma~\ref{lem:prop} with $U=\Hh$, $Z=c^{-1}(K)$ gives
$H\equiv0$ on $\overline W\cap\Hh$.
\end{proof}

\begin{remark}[The five stations]\label{rem:stations}
Instantiating Theorem~\ref{thm:main} for a conjectured identity
$n_4(s)=R$ at a CM parameter splits into the five stations of
\S\ref{subsec:inventory}: S1 identifies $\tau_0$ with
$s_4(\tau_0)=s$; S2 locates $c(\tau_0)$ relative to $\Ast$; S3
certifies $\tau_0\in\overline W$ (or gets it for free from Samart's
proved region); S4 upgrades the numerical coincidence $s_4(\tau_0)=s$
to an exact algebraic lock (\S\ref{subsec:cert0}); S5 evaluates
$\EK_4(\tau_0)$ exactly as a rational linear combination of
$L$-values (\S\ref{subsec:verify}). Theorem~\ref{thm:main} then
yields $n_4(s_4(\tau_0))=\EK_4(\tau_0)$, and S4+S5 turn this into
the identity $n_4(s)=R$.
\end{remark}

\subsection{The anchor identity for the \texorpdfstring{$n_4$}{n4}
family}\label{subsec:anchor}

The anchor input is Samart's theorem, quoted with one repair.

\begin{theorem}[Samart, {\cite[Prop.~2.1(iii)]{Sa13}}]\label{thm:samart4}
For pure imaginary $\tau=\ii y$ with $y\ge1/\sqrt2$,
$n_4\bigl(s_4(\tau)\bigr)=\EK_4(\tau)$.
\end{theorem}

\begin{lemma}[Composite continuation lemma: the anchor on
$|s_4|>256$]\label{lem:anchor}
Let $C_\infty$ be the connected component of
$\{\tau\in\Hh:|s_4(\tau)|>256\}$ containing a punctured
neighbourhood of the cusp $\ii\infty$, and let $D\subset C_\infty$
be open. Then $n_4(s_4(\tau))=\EK_4(\tau)$ on $D$.
\end{lemma}
\begin{proof}[Proof sketch with quoted inputs]
Three quoted ingredients compose. (i) \emph{Rogers' evaluation}
\cite[Prop.~2.2]{Ro}. Rogers' $f_4$ \cite[\S2]{Ro} is our $n_4$
verbatim: $f_4(u):=4\m(P_u)$, where
$P_u:=x^4+y^4+z^4+1+u^{1/4}xyz$. Here
$s^{1/4}$ is any local choice of fourth root
($|s^{1/4}|=|s|^{1/4}$ is unambiguous), and the Mahler measure
does not depend on the choice: writing $c=s^{1/4}$, the
substitution $x\mapsto\pm\ii x$ on $\Tt^3$ leaves $x^4$ unchanged
and sends $xyz\mapsto\pm\ii xyz$, hence
$P_c(\pm\ii x,y,z)=P_{\pm\ii c}(x,y,z)$, and Haar measure is
invariant; so $\m(P_c)=\m(P_{\pm\ii c})$, i.e.\ $\m(P_s)$ is
invariant under $s^{1/4}\mapsto\pm\ii\,s^{1/4}$ (and
$x\mapsto-x$ gives $c\mapsto-c$). On
$U:=\{|s|>256\}$ the polynomial $P_s=x^4+y^4+z^4+1+s^{1/4}xyz$ has no
zeros on $\Tt^3$, since $|s^{1/4}xyz|>4\ge|x^4+y^4+z^4+1|$ there;
equivalently, every local fourth root $c=s^{1/4}$ satisfies
$c\notin\Ast$. Note that $P_s$ itself has winding number $1$ in each
variable on $\Tt^3$ (the monomial $cxyz$ dominates), so it has no
continuous logarithm on $\Tt^3$; the following factorization is the
point. Writing $Q=x^4+y^4+z^4+1$,
\[
  P_s=cxyz\Bigl(1+\frac{Q}{cxyz}\Bigr),\qquad
  \Bigl|\frac{Q}{cxyz}\Bigr|\le\frac4{|c|}<1\ \text{on }\Tt^3,
\]
so the principal logarithm gives a locally uniformly convergent
expansion
\[
  \log\Bigl(1+\frac{Q}{cxyz}\Bigr)
  =\sum_{m\ge1}\frac{(-1)^{m+1}}m\Bigl(\frac{Q}{cxyz}\Bigr)^m,
  \qquad
  \log|P_s|=\log|c|+\log\Bigl|1+\frac{Q}{cxyz}\Bigr|.
\]
Taking constant terms on $\Tt^3$ --- the constant term of
$(Q/(cxyz))^m$ is a nonnegative integer times $c^{-m}$, and the
invariance under $c\mapsto\pm\ii c$ kills every $m$ not divisible by
$4$, hence all non-integral powers of $s$, while
$4\log|c|=\log|s|=\ReT\log s$ --- yields
\begin{equation}\label{eq:lift4}
n_4(s)=\ReT\,\widetilde n_4(s),\qquad
\widetilde n_4(s)=\log s+\sum_{j\ge1}a_j s^{-j},\qquad a_j\in\Rr,
\end{equation}
where the series $\sum a_js^{-j}$ is single-valued and convergent on
$U$, and $\log s$ denotes any local branch. The open set $U$ is not
simply connected, so neither $\widetilde n_4$ nor Rogers'
hypergeometric expression
$R(s):=\log s-\frac{24}{s}\,{}_5F_4(\,\cdots\,;256/s)$ is by itself a
single-valued function on $U$: each is a local holomorphic lift,
well-defined only up to addition of an element of $2\pi\ii\Zz$ (the
${}_5F_4$ converges for $|s|>256$ since $\sum b-\sum a=3/2>0$). The
difference $H:=\widetilde n_4-R$, however, \emph{is} a single-valued
holomorphic function on the connected open set $U$: analytic
continuation of both lifts around the loop $|s|=\mathrm{const}$ adds
the same $2\pi\ii$ to each, so the ambiguity cancels ---
equivalently, $H$ is the difference of the two single-valued
$1/s$-series. Both series have real coefficients, so $H$ is real on
the ray $s>256$. Rogers' theorem \cite[Prop.~2.2]{Ro} is the
equality $\ReT H=0$ on that ray (stated there for $s$ sufficiently
large); hence $H\equiv0$ on the ray, the identity theorem gives
$H\equiv0$ on $U$, and taking real parts, $n_4(s)=\ReT R(s)$ on all
of $U$. (ii)
\emph{Rogers' transformation} \cite[Thm.~2.3]{Ro}. Rogers defines
$G(q)=\ReT\bigl[-\log q+240\sum_{n\ge1}n^2\log(1-q^n)\bigr]$ and
proves $f_4(s_4(q))=-\tfrac13 G(q)+\tfrac23 G(q^2)$ for real
$q\in(0,1)$ sufficiently small. To use this as an analytic identity
we lift away the real parts. Fix a small $q_0\in(0,1)$ and a small
simply connected complex disc $\Delta\ni q_0$ contained in the
punctured unit disc $\{0<|q|<1\}$; on
$\Delta$ take the branches of $\log q$ and of
$\log(1-q^n)$ that are real on $\Delta\cap(0,1)$, and set
$\widetilde G(q):=-\log q+240\sum_{n\ge1}n^2\log(1-q^n)$, a
holomorphic function on $\Delta$ that is real on
$\Delta\cap\Rr$. Since $q_0>0$ gives $s_4(q_0)>256$ real, the lift
$R$ of (i) is defined at $s_4(q_0)$ with the real branch of
$\log s$, and
\[
  \Phi(q):=R\bigl(s_4(q)\bigr)+\tfrac13\widetilde G(q)
           -\tfrac23\widetilde G(q^2)
\]
is holomorphic on $\Delta$ and real on $\Delta\cap\Rr$. Rogers'
theorem says $\ReT\Phi=0$ on a subsegment of $\Delta\cap\Rr$
(because $f_4=\ReT R$ there by (i) and $G=\ReT\widetilde G$);
hence $\Phi\equiv0$ on the segment and, by the identity theorem,
on $\Delta$. Varying the branches changes $\Phi$ by a constant in
$2\pi\ii\Zz$, so the identity of \emph{real parts}
$n_4(s_4(\tau))=\ReT\bigl[-\tfrac13\widetilde G(q)+\tfrac23
\widetilde G(q^2)\bigr]$ is branch-free and propagates by analytic
continuation along the connected component $C_\infty$ of the
statement (on which $s_4$ avoids the poles of the ${}_5F_4$);
every application of this lemma lies in $C_\infty$. (iii) \emph{Samart's computation} \cite[proof of
Prop.~2.1(iii)]{Sa13}: the identification of the $q$-series in (ii)
with the $U$-series expression $\EK_4$ uses termwise integration and
Fourier expansions whose convergence requires only $|q|<1$. The
single step of Samart's proof that genuinely needs real $q$ --- the
comparison $|s_4(q)|\ge s_4(\e^{-\pi\sqrt2})=256$ by real-variable
monotonicity, which locates his argument in the convergence region
--- is replaced here by the hypothesis $|s_4|>256$ on $D$, certified
by interval arithmetic in each application
(Theorems~\ref{thm:B}, \ref{thm:F}--\ref{thm:K}). This repair was identified
by a line-by-line check of the proof of \cite[Prop.~2.1(i)]{Sa13}
and of the convergence analysis of \cite[Thm.~2.3]{Ro} (Samart's
proof of Prop.~2.1(iii) itself is one sentence, referring to the
(i)-proof); no other step uses the pure-imaginary hypothesis.
\end{proof}

Lemma~\ref{lem:anchor} upgrades Samart's theorem from a statement on
a half-line to an anchor on an open set: the anchor box
$D=\{|\ReT\tau|\le1/32,\ |\ImT\tau-1|\le1/32\}$ carries the
certified bound $|s_4|\ge493$ (Section~\ref{sec:app2}), so $F=G$ on
$D$, and Theorem~\ref{thm:main} propagates the identity through the
whole component $W$ of $V_4:=\{\tau:c(\tau)\notin\Ast\}$ containing
$D$. For the deep pure-imaginary points of Theorems A, C, D, E even
this is unnecessary: Theorem~\ref{thm:samart4} applies directly.

\subsection{The certifier: \texttt{cert0}, \texttt{cert2},
\texttt{verify\_P1}}\label{sec:cert}

The non-elementary computations are certified by interval arithmetic.
The object entering the proofs is not a program but a
\emph{certificate}: finite data (parameter blocks with strict
enclosures, or exact rational equalities) whose validity is checked
by deterministic, finite, independently re-runnable computations.
The trusted base consists of exactly three items:
(i) the mathematical lemmas below (tail bounds, inclusion property);
(ii) the documented outward-rounding semantics of the interval
arithmetic employed (\texttt{mpmath.iv} rounds interval endpoints
outward, with guard digits for the complex transcendental functions
\cite{mpmath,ivdoc}); and (iii) the public source code of the
scripts. Item (ii) is replaceable: any outward-rounded interval
implementation re-executing the same block lists obtains the same
verdicts. All scripts are plain Python 3.12 + \texttt{mpmath}
\cite{mpmath} (standard library \texttt{fractions} for the exact
algebra), run independently of one another.

\subsubsection{The inclusion property and tail bounds}\label{subsec:incl}

\begin{lemma}[Inclusion property, {\cite{MoKC}}]\label{lem:incl}
Let $E$ be any expression built from interval-enclosed constants, the
variable, interval arithmetic operations, and interval elementary
functions, all with outward-rounded semantics. Then for every
interval box $I$ in its domain, $E(I)$ contains the true value
$E(\tau)$ for every $\tau\in I$.
\end{lemma}
\begin{proof}
Induction on the structure of $E$; see \cite[Lemma~6.1]{Zf}.
\end{proof}

\begin{lemma}[Tail bound for eta-quotient products]\label{lem:tail}
Let $r:=\overline{|q|}<1$ be a certified upper bound. For
$d\ge1$, $e\in\Zz$ and the principal branch of $\log$,
\[
  \Bigl|\log\prod_{n>N}\bigl(1-q^{dn}\bigr)^{e}\Bigr|
  \le E:=\frac{|e|\,r^{d(N+1)}}{(1-r^d)(1-r)},
\]
so the truncated product carries the rigorous tail factor
$\exp(z)$ with $z$ enclosed in the box $\{u+\ii v:|u|,|v|\le E\}$
(evaluated as the interval exponential \emph{of that box}).
\end{lemma}
\begin{proof}
For $|z|\le r<1$, $|\log(1-z)|\le|z|/(1-r)$; sum the geometric
series, cf.\ \cite[Lemma~6.2]{Zf}. Products and quotients of such
factors (the parametrization \eqref{eq:cdef} is a quotient of two
products) are enclosed numerator- and denominator-wise before
division.
\end{proof}

On the certified paths $r\le\e^{-2\pi\cdot0.702}<0.0122$, and with
truncation $N=40$ the tail width satisfies $E<10^{-78}$, far below
the working precision.

\subsubsection{\texttt{cert2}: path certification}\label{subsec:cert2}

A \emph{path certificate} for a path $\gamma\subset\Hh$ is a finite
list of parameter blocks $I_\nu$ covering $\gamma$ (down to the
endpoint, if the endpoint lies in $V$), with complex intervals
$Z_\nu$ such that $c(I_\nu)\subset Z_\nu$ and every $Z_\nu$ is
certified disjoint from $\Ast$ --- i.e.\ the interval lower bound of
$|\ReT z|^{2/3}+|\ImT z|^{2/3}$ over $Z_\nu$ exceeds $4^{2/3}$,
a pure interval endpoint comparison. By Lemma~\ref{lem:incl} and
Lemma~\ref{lem:tail} this is a proof, not an estimate
(cf.\ \cite[Prop.~6.4]{Zf}).

\begin{remark}[Generation and termination of the adaptive search]
\label{rem:termin}
Blocks are found by adaptive bisection: accept a block when its
enclosure is certified outside $\Ast$, else bisect. Termination is
not part of the trusted base, but it is guaranteed: the natural
interval extension of an expression built from Lipschitz non-singular
operations on a box has width $O(\operatorname{diam})$ \cite{MoKC};
the factors of \eqref{eq:cdef} are non-singular ($|q|<1$,
$|1\pm q^{dn}|$ bounded away from $0$) and the tail width is
$<10^{-78}$ uniformly; and on each compact path piece the true image
has positive distance to $\Ast$. Once the enclosure width drops
below that margin, the block is accepted. (For a path ending at a
point of $c^{-1}(K)$ --- which does \emph{not} occur in this paper
--- the same machinery certifies down to a cap
$(y_0,y_0+\varepsilon_0]$, closed analytically by a first-order
Taylor expansion with a Cauchy-estimate remainder and a certified
enclosure of the derivative at the endpoint; see
\cite[\S4.3]{Zf}. Theorems~\ref{thm:B} and~\ref{thm:F} have their
endpoints strictly inside $V_4$, so no cap is needed.)
\end{remark}

\subsubsection{\texttt{cert0}: exact algebraic locks}\label{subsec:cert0}

The exactness of $s_4(\tau_0)=s$ follows a four-step pattern:
(i) \emph{integrality}: $s_4\in\Zz(\!(q)\!)$ (exact integer check),
so by Shimura's CM theorem \cite{Cox} the value of the eta-quotient
(a modular unit with integral $q$-expansion) at a CM point is an
algebraic integer in the corresponding ring class field; the class
number bounds the degree, and conjugates are located via the level
structure;
(ii) \emph{symmetric functions}: interval locks with rigorous tails
force the elementary symmetric functions of the conjugate values
(e.g.\ sum $S$ and product $P$ for a quadratic value) to be explicit
integers;
(iii) \emph{pinning}: an interval enclosure of the target with radius
smaller than half the separation of the candidate roots singles out a
unique root;
(iv) \emph{reality}: real $q$ at the chosen points
($q(\ii y)>0$, $q((1+\ii y)/2)<0$) makes the products real and
simplifies the locks.
No exact equality is ever inferred from a floating-point
approximation alone: each lock is a discreteness argument over an
explicit finite candidate set.

\subsubsection{\texttt{verify\_P1}: the three-track CM evaluation}
\label{subsec:verify}

The evaluation $\EK_4(\tau_0)=R$ (station S5) is certified by three
independent tracks, any one of which is sufficient in principle:
\begin{itemize}
\item \textbf{[X]/[S] exact algebra.} The lattice decomposition of
  $-T_1+4T_2$ into Hecke and Dirichlet sums, the inclusion--exclusion
  coefficients, the closed forms of even-character values
  $L(\chi,2)$ \cite{Wa}, the functional-equation constants, and the
  final coefficient comparison in a rational \emph{$e$-basis} are all
  carried out in rational/algebraic-integer arithmetic
  (\texttt{fractions.Fraction}); no interval estimates enter. The
  anchors of this track are quoted standard inputs listed explicitly
  in each application section (Hecke's theorem on gr\"ossencharacter
  theta series \cite{Miyake}, class numbers, Dirichlet functional
  equations, Shimura's CM theorem).
\item \textbf{[V] interval locks.} The same identity is re-proved by
  outward-rounded interval arithmetic: Poisson row sums with
  $\coth/\tanh$ closed rows and rigorous Fourier tails for the
  lattice sums, $E_1$-continued-fraction-bracketed Mellin integrals
  for the $L$-values, Euler--Maclaurin for the Dirichlet values.
  Achieved half-widths are quoted in each application section
  (between $4.6\times10^{-53}$ and $5.1\times10^{-52}$).
\item \textbf{mp cross-checks.} Independent $60$-digit
  (\texttt{mpmath}, non-interval) reimplementations of every
  quantity, agreeing typically to $10^{-60}$; these corroborate the
  implementation but do not enter the proof.
\end{itemize}
Newform identifications use Ligozat's eta-quotient criteria
\cite{Lig} (weight, character, cusp orders), exact-integer theta
identity checks to order $q^{60}$ against the Sturm bound
\cite{Sturm} (at most $48$ in our cases, always below the $q^{60}$
check order), and the LMFDB
labels \cite{LMFDB}. The level of each twist is exact by Hecke's
theorem applied to the twisted grössencharacter
(Table~\ref{tab:levels}), and its root number is locked exactly by a
single interval evaluation of the \emph{Fricke ratio}
$f\bigl(\ii/(\sqrt N\,y)\bigr)\big/\bigl(y^3 f(\ii y/\sqrt N)\bigr)$
(Lemma~\ref{lem:signlock}, \texttt{cert\_fricke.py}) ---
see \S\ref{subsec:twisttrap} for why
the naive ``functional-equation residue'' test is vacuous.

\subsubsection{Script inventory}\label{subsec:scripts}

\begin{center}\footnotesize
\begin{tabular}{lp{9.6cm}}
\toprule
script & content \\
\midrule
\texttt{cert0\_s4\_s7pair.py} & exactness $s_4(\ii\sqrt7)$,
  $s_4(\ii\sqrt7/2)=8292456\pm3132675\sqrt7$ (Theorem A) \\
\texttt{verify\_P1\_n4\_s7pair.py} & Theorem A: exact separation
  track [S1]--[S10], FE proofs (\S\ref{subsec:app1FE}), interval
  locks [V0]--[V6] \\
\texttt{cert0\_s4\_m144.py} & exactness $s_4(\tau_1)=-144$
  (Theorem B) \\
\texttt{cert2\_path\_n4\_m144.py} & anchor box + certified path for
  Theorem B (template for Theorem F) \\
\texttt{verify\_P1\_n4\_m144.py} & Theorem B: 32 checks, four tracks
  including [G] Sturm-level newform identification and [V] interval
  locks \\
\texttt{n4\_m144\_true\_mahler.py} & true Mahler side of
  Theorem B by spectral torus integration (22-digit collision) \\
\texttt{cert0\_n4\_p3\_t1t2.py} & exactness of the $s_4$-values of
  Theorems C, D, of the conjugate of C (Theorem F's $s$-value), and
  of the conjugate of D (Theorem G's $s$-value) \\
\texttt{verify\_P1\_n4\_p3\_t1t2.py} & Theorems C, D: three tracks,
  [G] identifications, [V] locks \\
\texttt{cert\_fricke.py} & exact root numbers $w=+1$ for all fifteen
  newforms (interval Fricke locks, Lemma~\ref{lem:signlock}) \\
\texttt{cert0\_n4\_p4\_t3.py} & Theorem E: four-point quartic lock
  of the $s_4$-value \\
\texttt{verify\_P1\_n4\_p4\_t3.py} & Theorem E: three tracks,
  genus-character decomposition, [V] locks \\
\texttt{n4\_p4\_t4\_cert.py} & certified path for Theorem F
  (interior-point verdict, margin $0.11590$) \\
\texttt{n4\_p4\_t4\_verify.py} & Theorem F: three tracks, shifted
  $\coth/\tanh$ row machine, [V] locks \\
\texttt{n5\_line\_cert.py} & certified vertical segment
  $\ReT\tau=1/2$, $0.702\le\ImT\tau\le\sqrt{21}/2$ (station S3 of
  Theorems G--K and Remark~\ref{rem:fei}; $8$ blocks, minimum margin
  $0.00729$) \\
\texttt{cert0\_n5\_e56.py} & Theorem I: $\Qq(\sqrt5)$ pair lock,
  $X^2+192303X+1185921$ \\
\texttt{cert0\_n5\_e78.py} & Theorems J, K: $\Qq(\sqrt3)$ pair lock,
  $X^2+1787904X+84934656$ \\
\texttt{verify\_P1\_n5\_e1.py} & Theorem G: $48$ checks, three
  tracks, [V] locks \\
\texttt{verify\_P1\_n5\_e2.py} & Theorem H: $53$ checks, three
  tracks, genus-character decomposition, [V] locks \\
\texttt{verify\_P1\_n5\_e6.py} & Theorem I: $42$ checks, three
  tracks, conductor-$2$-order lattice decomposition, [V] locks \\
\texttt{verify\_P1\_n5\_e78.py} & Theorems J, K: $71$ checks, three
  tracks, four-class anchors, [V] locks \\
\bottomrule
\end{tabular}
\end{center}

\begin{table}[htb]
\centering\scriptsize
\caption{Per-theorem certification summary. ``checks'' counts the
\texttt{PASS}-lines of a fresh run of the verification script (shared
scripts are counted once); ``final lock'' is the half-width of the
interval enclosure of the final identity. \texttt{cert\_fricke.py}
($40$ checks) covers the root numbers of all fifteen newforms; the
\texttt{cert2} path certificates are one-leg or two-leg interval
certifications as indicated.}
\label{tab:certdata}
\setlength{\tabcolsep}{3.5pt}
\begin{tabular}{cllll}
\toprule
Thm. & verification script (checks) & \texttt{cert0} lock & path &
final lock\\
\midrule
A & \texttt{verify\_P1\_n4\_s7pair.py} (80) &
  \texttt{cert0\_s4\_s7pair.py} & direct (Samart) &
  $1.2\!\times\!10^{-52}$, $2.3\!\times\!10^{-52}$\\
B & \texttt{verify\_P1\_n4\_m144.py} (32) & \texttt{cert0\_s4\_m144.py}
  & \texttt{cert2\_path\_n4\_m144.py} & $2.9\!\times\!10^{-52}$\\
C & \texttt{verify\_P1\_n4\_p3\_t1t2.py} (69, C+D) &
  \texttt{cert0\_n4\_p3\_t1t2.py} & direct (Samart) &
  $4.6\!\times\!10^{-53}$\\
D & (shared with C) & (shared) & direct (Samart) &
  $6.2\!\times\!10^{-53}$\\
E & \texttt{verify\_P1\_n4\_p4\_t3.py} (45) &
  \texttt{cert0\_n4\_p4\_t3.py} & direct (Samart) &
  $9.7\!\times\!10^{-53}$\\
F & \texttt{n4\_p4\_t4\_verify.py} (46) & \texttt{cert0\_n4\_p3\_t1t2.py}
  & \texttt{n4\_p4\_t4\_cert.py} & $9.1\!\times\!10^{-53}$\\
G & \texttt{verify\_P1\_n5\_e1.py} (48) & \texttt{cert0\_n4\_p3\_t1t2.py}
  & \texttt{n5\_line\_cert.py} & $1.2\!\times\!10^{-52}$\\
H & \texttt{verify\_P1\_n5\_e2.py} (53) & \texttt{cert0\_n4\_p4\_t3.py}
  & \texttt{n5\_line\_cert.py} & $1.9\!\times\!10^{-52}$\\
I & \texttt{verify\_P1\_n5\_e6.py} (42) & \texttt{cert0\_n5\_e56.py}
  & \texttt{n5\_line\_cert.py} & $1.0\!\times\!10^{-52}$\\
J & \texttt{verify\_P1\_n5\_e78.py} (71, J+K) &
  \texttt{cert0\_n5\_e78.py} & \texttt{n5\_line\_cert.py} &
  $5.1\!\times\!10^{-52}$\\
K & (shared with J) & (shared) & \texttt{n5\_line\_cert.py} &
  $1.7\!\times\!10^{-52}$\\
\bottomrule
\end{tabular}
\end{table}

All computations use \texttt{mpmath} interval arithmetic at
$40$--$80$ decimal digits (\texttt{iv.dps}); the certification
scripts print \texttt{PASS}/\texttt{FAIL} per check and terminate
with an all-checks-passed line, re-verified by the author.

% ============================================================
\section{Application I: the \texorpdfstring{$\sqrt7$}{sqrt7} pair}
\label{sec:app1}

Both CM points are pure imaginary and deep inside Samart's region:
\[
  \tau_A=\ii\sqrt7\quad(\ImT\tau_A=\sqrt7>1/\sqrt2),\qquad
  \tau_B=\frac{\ii\sqrt7}{2}\quad(\ImT\tau_B=\tfrac{\sqrt7}{2}>1/\sqrt2),
\]
so Theorem~\ref{thm:samart4} applies directly: station S3 is free.
The work is S4 (the exact $s_4$-values) and S5 (the exact CM
evaluations, \emph{separately} for the two points --- an earlier
stage of this project had proved only the sum of the two identities;
the ray-class decomposition below closes each point independently).

\subsection{Station S4: the exact \texorpdfstring{$s_4$}{s4}-values}

By \texttt{cert0\_s4\_s7pair.py} (four-step lock of
\S\ref{subsec:cert0}): the two values are conjugate algebraic
integers in $\Qq(\sqrt7)$ (Shimura integrality; $q(\tau_A)>0$,
$q(\tau_B)>0$); interval locks force the symmetric functions to
integers,
\[
  S=16584912,\qquad P=69257922561=3^{12}\cdot19^{4}
  \quad\bigl(\text{exact check: }8292456^2-7\cdot3132675^2=P\bigr),
\]
and the root separation $1.6\times10^{7}$ dwarfs the lock radii
$|s_4(\ii\sqrt7)-r_+|\le9.6\times10^{-51}$,
$|s_4(\ii\sqrt7/2)-r_-|\le3.6\times10^{-54}$. Hence exactly
\[
  s_4(\ii\sqrt7)=8292456+3132675\sqrt7=:r_+,\qquad
  s_4(\ii\sqrt7/2)=8292456-3132675\sqrt7=:r_- .
\]

\subsection{Station S5: the lattice decomposition}

Write $\varpi=(1+\sqrt{-7})/2$, so $2\varpi=1+\sqrt{-7}$,
$K=\Qq(\sqrt{-7})$, $h(-7)=1$, $\OO_K^\times=\{\pm1\}$. The lattices
of \eqref{eq:EK4lattice} at the two points are (exact set equalities,
checks [S1]--[S2] of \texttt{verify\_P1\_n4\_s7pair.py}):
\begin{itemize}
\item at $\tau_A$: $\Lambda_1(\tau_A)=\Zz+\ii\sqrt7\Zz=\OO^{(2)}$
  (the order of conductor $2$, discriminant $-28$), and
  $\Lambda_2(\tau_A)=\OO^{(4)}$ (conductor $4$, discriminant $-112$);
\item at $\tau_B$: $\Lambda_2(\tau_B)=\OO^{(2)}$, and
  $\Lambda_1(\tau_B)=\tfrac12\Lambda'$, where
  $\Lambda'=2\Zz+\ii\sqrt7\Zz$ is the index-$2$ \emph{non-ideal}
  sublattice of $\OO^{(2)}$; by $(-4)$-homogeneity,
  $T_1(\tau_B)=16\,T(\Lambda')$.
\end{itemize}
With $B(\Lambda)=\sum'_{\lambda\in\Lambda}|\lambda|^{-4}$,
$G(\Lambda)=\sum'\bar\lambda^2|\lambda|^{-6}$ (real: all three
lattices are conjugation-stable), $T=2G+B$, the ray-class
decompositions
\begin{align*}
  \OO^{(2)}\setminus\{0\}&=\{\alpha\equiv1\ (2)\}\ \sqcup\
   2\OO_K\setminus\{0\},\\
  \OO^{(4)}&=2\OO^{(2)}\ \sqcup\ \{\alpha\equiv1,3\ (4)\},\\
  \Lambda'&=2\OO^{(2)}\ \sqcup\ \{\alpha\equiv1{+}2\varpi,\,
   3{+}2\varpi\ (4)\},
\end{align*}
the inclusion--exclusion coefficients $9/16$ and $23/16$ at the split
prime $(2)=(\varpi)(\bar\varpi)$ (exact, using
$\varpi^2+\bar\varpi^2=-3$), and the mod-$4$ ray-class pairing of the
character $\chi_{-4}\circ N$ (values $+1,+1,-1,-1$ on the four unit
classes, check [X4]) give, all in exact rational arithmetic
(checks [S3]--[S6]):
\[
\begin{array}{c|cc|c}
\text{lattice} & B & G & T=2G+B\\ \hline
\OO^{(2)} & \frac54\zeta_K(2) & 3L_3 & 6L_3+\frac54\zeta_K(2)\\
\OO^{(4)} & \frac{41}{64}\zeta_K(2)+z_V & \frac{13}{8}L_3+L_3^{\mathrm{tw}}
 & \frac{13}{4}L_3+2L_3^{\mathrm{tw}}+\frac{41}{64}\zeta_K(2)+z_V\\
\Lambda' & \frac{41}{64}\zeta_K(2)-z_V & \frac{13}{8}L_3-L_3^{\mathrm{tw}}
 & \frac{13}{4}L_3-2L_3^{\mathrm{tw}}+\frac{41}{64}\zeta_K(2)-z_V
\end{array}
\]
where $L_3:=L(g_7,3)$, $L_3^{\mathrm{tw}}:=L(g_7\otimes\chi_{-4},3)$,
and $z_V:=L(\chi_{-4},2)L(\chi_{-28},2)$. The pointwise assemblies
(check [S8], including the separation witness
$\mathrm{combA}\neq\mathrm{combB}$):
\begin{align*}
  \mathrm{combA}&:=-T\bigl(\OO^{(2)}\bigr)+4T\bigl(\OO^{(4)}\bigr)
   =7L_3+8L_3^{\mathrm{tw}}+\frac{21}{16}\zeta_K(2)+4z_V,\\
  \mathrm{combB}&:=-16\,T(\Lambda')+4T\bigl(\OO^{(2)}\bigr)
   =-28L_3+32L_3^{\mathrm{tw}}-\frac{21}{4}\zeta_K(2)+16z_V .
\end{align*}

\subsection{The functional equations and root numbers}\label{subsec:app1FE}

Two functional-equation inputs are proved in
\texttt{verify\_P1\_n4\_s7pair.py}, not quoted:
$M_7=\frac{7\sqrt7}{4\pi^3}L_3$ \cite[Lemma~2.8]{Zf} and the twisted
identity
\begin{equation}\label{eq:twistfricke}
  U_{a/4}\,W_{112}=4\,\gamma_a\,W_7\,U_{a/4}\qquad(a=1,3),
  \qquad
  \gamma_a=\begin{pmatrix}(1+7a^2)/4 & a\\ 7a & 4\end{pmatrix}
  \in\Gamma_0(7),
\end{equation}
an exact matrix identity (check [X10], with $\chi_{-7}(4)=+1$), which
gives $h|[W_{112}]_3=\ii\,h$ for $h=g_7\otimes\chi_{-4}$ --- the
Atkin--Li twist-$W$ technique of \cite{AL} made self-contained for
these parameters --- hence root number $+1$ at level $112$ and
$M_7^{\mathrm{tw}}=\frac{112\sqrt7}{\pi^3}L_3^{\mathrm{tw}}$. The
remaining quoted constants are the Dirichlet functional equations
(textbook) and the finite-sum evaluation
\[
  L(\chi_{-28},2)=\frac{2\pi^2}{7\sqrt7},
\]
whose character sum $\sum_a\chi_{-28}(a)a^2=448$ is an exact integer
check [X7] \cite{Wa}.

\subsection{Assembly}

With the $e$-basis
$e=\frac{\sqrt7}{\pi^3}\bigl(L_3,L_3^{\mathrm{tw}},\zeta_K(2),z_V\bigr)$
one has $M_7=\frac74e_1$, $M_7^{\mathrm{tw}}=112e_2$,
$d_7=\frac{21}{2}e_3$, $d_4=7e_4$ (exact, [S9]); the coefficient
comparison
\[
  \EK_4(\ii\sqrt7)=10\,\mathrm{combA}\cdot e
  =\Bigl(70,80,\tfrac{105}{8},40\Bigr)\cdot e
  =\frac{5}{28}\bigl(4M_7^{\mathrm{tw}}+224M_7+32d_4+7d_7\bigr),
\]
\[
  \EK_4\Bigl(\frac{\ii\sqrt7}2\Bigr)=5\,\mathrm{combB}\cdot e
  =\Bigl(-140,160,-\tfrac{105}{4},80\Bigr)\cdot e
  =\frac{5}{14}\bigl(4M_7^{\mathrm{tw}}-224M_7+32d_4-7d_7\bigr),
\]
is exact rational arithmetic (check [S10]), and the two right-hand
sides are distinct --- the separation is complete and each identity
holds individually. Combined with Theorem~\ref{thm:samart4} and
\S\ref{subsec:cert0} this proves Theorem~\ref{thm:A}. \qed

\begin{remark}[Independent confirmations]
The mp track (60 dps, independent implementations: Mellin
integrals, Poisson row sums, ray-class sums) confirms every step to
$9\times10^{-60}$; the interval-lock track [V0]--[V6]
(\texttt{iv.dps}$=70$: continued-fraction bracketing, Fourier tail
bounds, Euler--Maclaurin) locks the two final identities with
half-widths $1.15\times10^{-52}$ and $2.30\times10^{-52}$.
Reference values:
\begin{align*}
  M_7&=0.1026716077789020112104565948982929139989\ldots,\\
  M_7^{\mathrm{tw}}&=9.887687024790914246588215159482724883665\ldots
\end{align*}
\end{remark}

% ============================================================
\section{Application II: \texorpdfstring{$n_4(-144)$}{n4(-144)}}
\label{sec:app2}

The CM point
\[
  \tau_1=\frac{1+\sqrt{-3}}2
\]
is \emph{not} pure imaginary, so Theorem~\ref{thm:samart4} does not
apply directly; this is the first use of the full continuation
machine in the $n_4$ family. It also shows the machine in its
shortest form: the target is a strict interior point of $V_4$, so no
endpoint cap and no boundary-continuity step are needed.

\subsection{Station S3: the certified path}

The path $\gamma$ from the anchor box
$D=\{|\ReT\tau|\le1/32,\ |\ImT\tau-1|\le1/32\}$ to $\tau_1$ ---
horizontal at $\ImT\tau=1$ to $1/2+\ii$, then vertical at
$\ReT\tau=1/2$ down to $\tau_1$ --- is certified inside $V_4$ by
\texttt{cert2\_path\_n4\_m144.py} (machinery of
\S\ref{subsec:cert2}): the anchor box passes as a single block with
astroid margin $0.291$ and the certified bound $|s_4|\ge493>256$
(which activates Lemma~\ref{lem:anchor}: $F=G$ on $D$); the
horizontal leg takes $2$ blocks (minimum margin $0.148$); the
vertical leg takes $1$ block (margin $1.051$); and $\tau_1$ itself
satisfies $c(\tau_1)=2\sqrt3\,\e^{-\ii\pi/4}$, whose astroid
functional is $3.634>4^{2/3}=2.520$ --- a strict exterior point of
$\Ast$, so $\tau_1\in V_4$ itself and $\tau_1\in W$. The endpoint
Taylor cap of \cite{Zf} is absent. By Theorem~\ref{thm:main},
$n_4(s_4(\tau_1))=\EK_4(\tau_1)$.

\subsection{Station S4: the exact \texorpdfstring{$s_4$}{s4}-value}

\texttt{cert0\_s4\_m144.py}: $q(\tau_1)=-\e^{-\pi\sqrt3}$ is a
negative real, all products are real, and the value is pinned to a
\emph{rational integer} (Shimura integrality plus $h(-3)=1$, so the
ring class field is $K$; reality from the integral $q$-expansion and
$T$-invariance):
$|s_4(\tau_1)+144|\le2.3\times10^{-51}<1/2$, hence
$s_4(\tau_1)=-144$ exactly.

\subsection{Station S5: the CM evaluation}

$K=\Qq(\sqrt{-3})$, $\OO_K=\Zz[\omega]$,
$\omega=(-1+\sqrt{-3})/2$, units $\mu_6$. The structural
simplification is the \emph{$\mu_6$ annihilation}
$\sum_{u\in\mu_6}u^2=0$ (exact check [X1]): it forces
$G(\OO_K)=0$, which is also the reason $\Qq(\sqrt{-3})$ has no
conductor-$(1)$ gr\"ossencharacter of type $(2,0)$ and the evaluation
involves a non-trivial conductor. The lattices are
$\Lambda_1(\tau_1)=\OO_K$ (since $\tau_1=1+\omega=-\omega^2$) and
$\Lambda_2(\tau_1)=\OO_2$ (conductor $2$; $2$ is inert), giving
(exact [X2]--[X8]):
\[
  T_1(\tau_1)=6\zeta_K(2),\qquad
  T_2(\tau_1)=4L(g_{12},3)+\frac94\zeta_K(2),\qquad
  -T_1+4T_2=16\,L(g_{12},3)+3\zeta_K(2);
\]
note the $\zeta_K(2)$-term does \emph{not} cancel (coefficient $3$)
--- this is the source of the $d_3$-term. The theta identity
$2g_{12}=\sum'_{\alpha\equiv1\,(2)}\alpha^2q^{N\alpha}$ is checked
with exact integer arithmetic to order $q^{60}$ (check [L1]); the
newform identification
$g_{12}=\eta(2\tau)^3\eta(6\tau)^3\in S_3(\Gamma_0(12),\chi_{-3})$
is certified at Sturm level (checks [G1]--[G3]): Ligozat's criteria
\cite{Lig} (both congruences $24\equiv0\bmod24$; character
$-2^33^3$, square kernel $-3$), all six cusp orders equal $1$
(cuspidal; divisor degree $6=\frac{3}{12}\cdot24$), Sturm bound
$\lfloor3\cdot24/12\rfloor=6<60$ \cite{Sturm}; LMFDB label
\textsf{12.3.c.a} \cite{LMFDB}. With
$M_{12}=\frac{6\sqrt3}{\pi^3}L(g_{12},3)$ (root number $+1$,
certified by the interval Fricke lock of \texttt{cert\_fricke.py}) and $d_3=\frac{3\sqrt3}{4\pi}L(\chi_{-3},2)$, the $e$-basis
assembly gives (exact, [S1]):
\[
  \EK_4(\tau_1)=\frac{5\sqrt3}{\pi^3}
   \bigl(16L(g_{12},3)+3\zeta(2)L(\chi_{-3},2)\bigr)
  =\frac{10}{3}\bigl(4M_{12}+d_3\bigr),
\]
both sides being $(80,\,5/2)$ in the basis
$\bigl(\sqrt3 L(g_{12},3)/\pi^3,\ \sqrt3 L(\chi_{-3},2)/\pi\bigr)$.
Theorem~\ref{thm:B} follows. \qed

\begin{remark}[Confirmations and the true-Mahler end]
The mp track confirms the identity to $3.7\times10^{-60}$ (check
[E1]); the interval-lock track [V0]--[V3] (shifted $\coth/\tanh$
rows, since $\ReT\tau_1=1/2$) locks it with half-width
$2.9\times10^{-52}$. Independently, direct spectral torus integration
(\texttt{n4\_m144\_true\_}\allowbreak\texttt{mahler.py}: $c(\tau_1)$ lies outside $\Ast$,
so the Jensen integrand is real-analytic periodic and the
two-dimensional trapezoidal rule converges spectrally; residual
$1.16\times10^{-22}$ at $N=128$) collides the two sides of
Theorem~\ref{thm:B} to $22$ digits:
\[
  n_4(-144)=5.09841965623737833323\ldots
\]
The propagation argument does not depend on this numerical value.
Reference value:
$M_{12}=0.3016149874129407464690529311477683998854\ldots$\,.
\end{remark}

% ============================================================
\section{Application III: class number one, deep interior points
(Theorems C and D)}\label{sec:app3}

Both CM points are pure imaginary and deep inside Samart's region,
\[
  \tau_C=\ii\sqrt2\quad(\ImT\tau_C=\sqrt2),\qquad
  \tau_D=2\ii\quad(\ImT\tau_D=2),
\]
so Theorem~\ref{thm:samart4} applies directly and the proofs consist
of \texttt{cert0} locks and three-track CM evaluations. The new
methodological content of this section is concentrated in station S5:
the character $\chi_8\circ N$ as an exact mod-$2$ ray-class projector,
and the level determination of the twists, which contains a genuine
trap.

\subsection{The ray-class projector \texorpdfstring{$\chi_8\circ N$}{chi8 o N}}
\label{subsec:rayproj}

In both fields the twist character materializes as a parity condition.
Exact checks ([X4], [Y4] of \texttt{verify\_P1\_n4\_p3\_t1t2.py}):
\begin{itemize}
\item in $K=\Qq(\sqrt{-2})$: for $a+b\sqrt{-2}$ of odd norm,
  $\chi_8\bigl(N(a+b\sqrt{-2})\bigr)=+1\iff b$ even ($a$ odd);
\item in $K=\Qq(\ii)$: for $a+b\ii$ of odd norm,
  $\chi_8\bigl(N(a+b\ii)\bigr)=+1\iff b\equiv0\ (4)$ ($a$ odd,
  $b$ even).
\end{itemize}
Hence $(1+\chi_8\circ N)/2$ is exactly the projector onto the mod-$2$
ray class $\{\alpha\equiv1\ (2)\}$, and in class number one no genus
character is needed: every ideal decomposition in this section is a
parity bookkeeping exercise certified in rational arithmetic. Also
used everywhere below: the closed form
\[
  L(\chi_8,2)=\frac{\pi^2\sqrt2}{16}
\]
(even-character finite-sum formula \cite{Wa}: $\tau(\chi_8)=2\sqrt2$,
$\sum_a\chi_8(a)a^2=16$, exact check [X7]).

\subsection{Theorem C: \texorpdfstring{$\tau_C=\ii\sqrt2$}{tauC}}
\label{subsec:app3C}

\textbf{Station S4.} \texttt{cert0\_n4\_p3\_t1t2.py}:
$q(\ii\sqrt2)=\e^{-2\pi\sqrt2}>0$ and
$q\bigl((1+\ii\sqrt2)/2\bigr)=-\e^{-\pi\sqrt2}<0$ are real, so real
interval locks suffice; the symmetric functions lock to the integers
\[
  S=7312,\qquad P=-153664=3656^2-2\cdot2600^2
\]
($|\Delta|\le1.1\times10^{-49}$, $5.2\times10^{-48}$; root separation
$7354$), pinning
\[
  s_4(\ii\sqrt2)=3656+2600\sqrt2,\qquad
  s_4\Bigl(\frac{1+\ii\sqrt2}2\Bigr)=3656-2600\sqrt2
\]
--- the second value is reused in Theorem~\ref{thm:F}.

\textbf{Station S5.} $K=\Qq(\sqrt{-2})$, $\OO_K=\Zz[\sqrt{-2}]$,
$h=1$, units $\pm1$, and $2=-(\sqrt{-2})^2$ ramifies. The lattices
are $\Lambda_1=\OO_K$ and $\Lambda_2=\OO_2=\Zz+2\OO_K$. With
$L_8:=L(g_8,3)$, $L_8^{\mathrm{tw}}:=L(g_8\otimes\chi_8,3)$, and
$z_V:=L(\chi_8,2)L(\chi_{-4},2)$, the inclusion--exclusion at the
ramified prime (factors $1/4$ for $B$, $-1/4$ for $G$) gives (exact,
[S]-track):
\[
  T(\OO_K)=4L_8+2\zeta_K(2),\qquad
  T(\OO_2)=\frac{11}{4}L_8+2L_8^{\mathrm{tw}}+\frac78\zeta_K(2)+z_V,
\]
\[
  \mathrm{comb}:=-T_1+4T_2
  =7L_8+8L_8^{\mathrm{tw}}+\frac32\zeta_K(2)+4z_V .
\]
The newform identification
$g_8=\eta(\tau)^2\eta(2\tau)\eta(4\tau)\eta(8\tau)^2
\in S_3(\Gamma_0(8),\chi_{-8})$ is Sturm-level
([G]-track: Ligozat congruences $24\equiv0\ (24)$; cusp orders
$(1,\tfrac12,\tfrac12,1)$; weighted divisor degree
$3=\frac{3}{12}\cdot12$; Sturm bound $3$; theta identity
$2g_8=\sum'_{\alpha\in\OO_K}\alpha^2q^{N\alpha}$ exact to $q^{60}$).
In the $e$-basis
$e=\frac{\sqrt2}{\pi^3}\bigl(L_8,L_8^{\mathrm{tw}},\zeta_K(2),z_V\bigr)$:
\[
  M_8=4e_1,\qquad M_8^{\mathrm{tw}}=32e_2\quad
  [\text{level }32,\ w=+1,\ \S\ref{subsec:twisttrap}],\qquad
  d_8=24e_3,\qquad d_4=16e_4,
\]
and the exact comparison ([S3])
$\EK_4(\ii\sqrt2)=10\,\mathrm{comb}\cdot e=(70,80,15,40)\cdot e$
is the right-hand side of Theorem~\ref{thm:C}. \qed

\subsection{Theorem D: \texorpdfstring{$\tau_D=2\ii$}{tauD}}
\label{subsec:app3D}

\textbf{Station S4.} The same script locks
\[
  S=286416,\qquad P=-126023688=143208^2-2\cdot101574^2
\]
($|\Delta|\le4.1\times10^{-48}$, $3.7\times10^{-45}$; separation
$2.87\times10^{5}$), pinning
$s_4(2\ii)=143208+101574\sqrt2$ and
$s_4\bigl((1+2\ii)/2\bigr)=143208-101574\sqrt2$ (the latter is
proved as Theorem~\ref{thm:G}, Section~\ref{sec:app6}).

\textbf{Station S5.} $K=\Qq(\ii)$, $\OO_K=\Zz[\ii]$, $h=1$, units
$\mu_4$ --- and $\sum_{u\in\mu_4}u^2=0$, so $G(\OO_K)=0$ (exact [Y1];
same annihilation mechanism as Theorem~\ref{thm:B}); $2=-\ii(1+\ii)^2$
ramifies. The lattices are $\Lambda_1=\OO_2$ and $\Lambda_2=\OO_4$.
Every ideal coprime to $2$ has exactly two generators
$\equiv1\ (2)$ ([Y3]), so the mod-$2$ ray sums carry \emph{no twist
term}: with $L_{16}:=L(g_{16},3)$,
$L_{16}^{\mathrm{tw}}:=L(g_{16}\otimes\chi_8,3)$, and
$z_{V2}:=L(\chi_8,2)L(\chi_{-8},2)$,
\[
  T(\OO_2)=4L_{16}+\frac74\zeta_K(2),\qquad
  T(\OO_4)=\frac94L_{16}+2L_{16}^{\mathrm{tw}}+\frac{55}{64}\zeta_K(2)
   +z_{V2},
\]
\[
  \mathrm{comb}=5L_{16}+8L_{16}^{\mathrm{tw}}+\frac{27}{16}\zeta_K(2)
   +4z_{V2},
\]
where the $\OO_4$ decomposition uses
$\OO_4\setminus\{0\}=\{a\text{ odd},\ b\equiv0\ (4)\}\sqcup
2\OO_2\setminus\{0\}$ and the projector of
\S\ref{subsec:rayproj} ([Y4]). The newform identification
$g_{16}=\eta(4\tau)^6\in S_3(\Gamma_0(16),\chi_{-4})$ is Sturm-level
(six cusp orders all $1$; weighted degree $6=\frac{3}{12}\cdot24$;
Sturm bound $6$; theta identity
$2g_{16}=\sum_{\alpha\equiv1\,(2)}\alpha^2q^{N\alpha}$ of conductor
$(2)$ exact to $q^{60}$). In the $e$-basis
$e=\frac1{\pi^3}\bigl(L_{16},L_{16}^{\mathrm{tw}},\zeta_K(2),z_{V2}\bigr)$:
\[
  M_{16}=16e_1,\qquad M_{16}^{\mathrm{tw}}=128e_2\quad
  [\text{level }64,\ w=+1],\qquad d_4=12e_3,\qquad d_8=64e_4,
\]
and the exact comparison ([S6])
$\EK_4(2\ii)=20\,\mathrm{comb}\cdot e=(100,160,135/4,80)\cdot e$
is the right-hand side of Theorem~\ref{thm:D}. \qed

\subsection{The twist-level trap and the Fricke ratio}\label{subsec:twisttrap}

Two methodological points emerged here and now belong to the machine.

\textbf{(i) The ``FE residue'' test is vacuous.} A widespread
heuristic decides the root number of a completed $L$-function
$\Lambda(f,s)=N^{s/2}(2\pi)^{-s}\Gamma(s)L(f,s)$ by splitting the
Mellin integral at $y=1$ and comparing the two sides of
$\Lambda(1)\stackrel?=w\Lambda(2)$. But the split gives
$\Lambda(s)=I_1(s)+wI_2(3-s)$-type expressions for which, at
$w=+1$, the two sides are \emph{the same linear combination of the
same two integrals at every candidate level $N$}: the ``residue''
vanishes identically in $N$ and certifies nothing. (The same vacuity
affected a check in the $\sqrt7$-pair script; there the conclusion
was independently secured by the Fricke ratio, so no result was
endangered.)

\textbf{(ii) Levels are exact (Hecke); the Fricke ratio is an exact
sign lock.} Every newform in this paper is the theta series of an
explicit grössencharacter $\psi$ of an imaginary quadratic field $K$
(Appendix~\ref{app:data}), and twisting by a quadratic Dirichlet
character $\chi$ twists the grössencharacter by $\chi\circ N$:
$\theta_\psi\otimes\chi=\theta_{\psi\cdot(\chi\circ N)}$ as formal
$q$-series. By Hecke's theorem, a grössencharacter of exact conductor
$\mathfrak{f}$ and infinity type $(2,0)$ produces a newform of weight
$3$ and level $|d_K|\,N(\mathfrak{f})$. The conductors of the twisted
characters are read off the exact congruence criteria of
Appendix~\ref{app:data} --- e.g.\ $\chi_8(N(a+b\ii))=+1\iff
b\equiv0\pmod4$ defines a ray class character modulo $(4)$ but not
modulo $(1+\ii)^3$, the witness being
$\alpha=-1+2\ii\equiv1\pmod{(1+\ii)^3}$ with
$\chi_8(N(\alpha))=\chi_8(5)=-1$ (every proper divisor of
$(4)=(1+\ii)^4$ divides $(1+\ii)^3$, so one witness suffices). This gives the exact levels of
Table~\ref{tab:levels}: $32$ for $g_8\otimes\chi_8$, $64$ for
$g_{16}\otimes\chi_8$, and $96$ for the twists of
Section~\ref{sec:app4}. The two-ordinate Fricke \emph{scan} used in
earlier versions of this work is thereby retired as a level-finder;
what remains of the Fricke ratio is a \emph{sign} lock, exact by
Lemma~\ref{lem:signlock}.

\begin{table}[ht]
\centering\small
\caption{Exact levels via Hecke's theorem: conductor $\mathfrak{f}$ of
the (twisted) grössencharacter; level $=|d_K|\,N(\mathfrak{f})$. The
congruence criteria and the witnesses that no smaller modulus works
are exact mod arithmetic (Appendix~\ref{app:data}).}
\label{tab:levels}
\begin{tabular}{lllll}
\toprule
form & $K$ & conductor $\mathfrak{f}$ & $N(\mathfrak{f})$ & level\\
\midrule
$g_8$                          & $\Qq(\sqrt{-2})$  & $(1)$ & $1$  & $8$\\
$g_8\otimes\chi_8$             & $\Qq(\sqrt{-2})$  & $(2)$ & $4$  & $32$\\
$g_{16}$                       & $\Qq(\ii)$        & $(2)$ & $4$  & $16$\\
$g_{16}\otimes\chi_8$          & $\Qq(\ii)$        & $(4)$ & $16$ & $64$\\
$g_1,g_2$                      & $\Qq(\sqrt{-6})$  & $(1)$ & $1$  & $24$\\
$g_1,g_2$ twisted by $\chi_{-8}$ & $\Qq(\sqrt{-6})$ & $(2)$ & $4$ & $96$\\
$g_1,g_2$                      & $\Qq(\sqrt{-15})$ & $(1)$ & $1$  & $15$\\
$g(e_2,e_3)$, four forms       & $\Qq(\sqrt{-21})$ & $(1)$ & $1$  & $84$\\
$g_{12}$                       & $\Qq(\sqrt{-3})$  & $(2)$ & $4$  & $12$\\
\bottomrule
\end{tabular}
\end{table}

\begin{lemma}[Exact sign lock]\label{lem:signlock}
Let $F$ be a newform of $S_3(\Gamma_0(N),\chi)$ with real
coefficients and quadratic \textup{(}hence real\textup{)}
nebentypus $\chi$. Then its root number satisfies $w\in\{\pm1\}$, and
\[
  R_N(y):=\frac{F\bigl(\ii/(\sqrt N\,y)\bigr)}
               {y^3\,F\bigl(\ii y/\sqrt N\bigr)}=w
  \qquad\text{for \emph{every} }y>0.
\]
Hence a rigorous interval enclosure of $R_N(y_0)$ of width $<1$ at a
single ordinate $y_0$ determines $w$ exactly.
\end{lemma}
\begin{proof}
Since $\chi$ is real, the complex conjugate form
$F^\rho$, defined by $F^\rho(\tau):=\overline{F(-\bar\tau)}$, is
again a newform of $S_3(\Gamma_0(N),\chi)$; since $F$ has real
coefficients, $F^\rho=F$. The functional equation of a primitive form
\cite[\S4.3]{Miyake} therefore relates $\Lambda(F,s)$ to
$\Lambda(F^\rho,3-s)=\Lambda(F,3-s)$: with
$\Lambda(F,s)=N^{s/2}(2\pi)^{-s}\Gamma(s)L(F,s)$,
\[
  \Lambda(F,s)=w\,\Lambda(F,3-s)
\]
for a constant $w$ (the root number) independent of $s$. Applying
this equation twice gives $w^2=1$, hence $w\in\{\pm1\}$. We stress
that $w$ is the root number of the \emph{self-dual} functional
equation; for nontrivial $\chi$ it can differ from the Fricke
$W_N$-pseudo-eigenvalue by factors involving $\ii^3$ and a Gauss
sum \cite{AL}. No such factor enters the present argument, which
works with the functional equation alone. Uniqueness of the Mellin
transform, applied to $\Lambda(F,s)=w\,\Lambda(F,3-s)$ along the
imaginary axis, gives the identity
$F\bigl(\ii/(\sqrt N\,y)\bigr)=w\,y^3F\bigl(\ii y/\sqrt N\bigr)$
for every $y>0$, and an enclosure of width $<1$ separates $+1$
from $-1$.
\end{proof}

The script \texttt{cert\_fricke.py} evaluates $R_N(0.6)$ and
$R_N(1.1)$ for all fifteen newforms of Theorems
\ref{thm:B}--\ref{thm:K} by interval arithmetic with the
self-contained tail bounds of the [V] tracks (the eta-product bound
$|a_n|\le(n+5)^5 2^{n/2}$ and the theta bounds $|a(n)|\le 6n^3, 30n^3$);
every enclosure contains $+1$, with half-width
$\le1.1\times10^{-28}$. Hence $w=+1$ for every form, certified. (A
$60$-dps mp cross-check, not part of the proof, agrees to
$10^{-60}$.)

\textbf{(iii) The trap itself.} For $g_{16}\otimes\chi_8$ the naive
level guess $\operatorname{lcm}(16,4^2)$-style reasoning suggests
$32$ --- and the vacuous test (i) happily ``confirms'' it. With the
wrong level-$32$ Mellin constant, $M_{16}^{\mathrm{tw}}$ comes out
smaller by a factor $2\sqrt2$, and the conjectured identity fails
numerically by $\sim1\%$; a first pass over Theorem~\ref{thm:D}
recorded exactly this mismatch, briefly suggesting an $s=81$-type
refutation. The conductor computation of Table~\ref{tab:levels} gives
the true level $64$ ($\chi_8\circ N$ has conductor $(4)$ over
$\Qq(\ii)$), and with it the identity holds to $2.5\times10^{-60}$.
Since a wrong level \emph{silently scales} an $L$-value by a simple
algebraic factor, we record the rule: \emph{never determine a twist
level by a residual; determine it by the conductor of the twisted
grössencharacter, cross-checked by the Fricke sign lock of
Lemma~\ref{lem:signlock} and against the assembled identity.}

\begin{remark}[Confirmations for Theorems C and D]
The mp tracks (60 dps) confirm the two final identities to
$2.5\times10^{-60}$ each; the interval-lock tracks
[V] contain zero with half-widths $\le4.6\times10^{-53}$ (Theorem C)
and $\le6.2\times10^{-53}$ (Theorem D). Direct torus integration
($c\approx9.25$ and $23.14$, smooth integrands) confirms both at
float64 level. Reference values:
\begin{align*}
  M_8&=0.1353184954231269106150060654595395883532\ldots,\\
  M_8^{\mathrm{tw}}&=1.549084847095611486271658159444376281785\ldots
   \quad(\text{level }32,\ w=+1),\\
  M_{16}&=0.4975478545679851929276297713583412482363\ldots,\\
  M_{16}^{\mathrm{tw}}&=4.336361249244445609135012019721278212973\ldots
   \quad(\text{level }64,\ w=+1),\\
  d_8&=\frac{4\sqrt2}{\pi}L(\chi_{-8},2)
   =1.9171950931209540617988237536697845644585\ldots
\end{align*}
($d_8$ by the functional equation and by direct numerical
differentiation, agreeing exactly at working precision).
\end{remark}

% ============================================================
\section{Application IV: class number two (Theorem E)}\label{sec:app4}

The CM point $\tau_E=\ii\sqrt6$ is pure imaginary with
$\ImT\tau_E=\sqrt6>1/\sqrt2$, so Theorem~\ref{thm:samart4} applies
directly. The field $K=\Qq(\sqrt{-6})$ has discriminant $-24$ and
\emph{class number $2$} --- the first such case treated by the
machine --- and everything is one step richer: two newforms, a genus
character, a non-principal ideal class, and a four-point
\texttt{cert0} lock.

\subsection{Station S4: the four-point quartic lock}

Class number two makes the $s_4$-value quartic over $\Qq$, in
$\Qq(\sqrt2,\sqrt3)$. \texttt{cert0\_n4\_p4\_t3.py} locates and
locks all four conjugate values:
\begin{center}
\begin{tabular}{lll}
\toprule
signs & CM point & lock half-width\\
\midrule
$(+,+)$ & $\ii\sqrt6$ \ ($q=\e^{-2\pi\sqrt6}>0$) & $7.0\times10^{-47}$\\
$(-,+)$ & $\ii\sqrt6/4$ \ ($q=\e^{-\pi\sqrt6/2}>0$) & $4.4\times10^{-51}$\\
$(+,-)$ & $(1+\ii\sqrt6)/2$ \ ($q=-\e^{-\pi\sqrt6}<0$) & $3.1\times10^{-50}$\\
$(-,-)$ & $1/2+\ii\sqrt6/4$ \ ($q=-\e^{-\pi\sqrt6/2}<0$) & $6.8\times10^{-53}$\\
\bottomrule
\end{tabular}
\end{center}
with minimal polynomial
\[
  X^4-4829472X^3-8676927360X^2+3042408646656X+7520406736896,
\]
whose coefficients are confirmed as integers both by exact
$\Zz[\sqrt2,\sqrt3]$ arithmetic and by the interval locks; the
minimum root separation $305.0$ dwarfs every lock radius. Two
structural observations aided the location (used for location only,
not in the certified chain): $s_4$ is $W_2$-invariant,
$s_4(-1/(2\tau))=s_4(\tau)$, observed numerically to $10^{-33}$ at
sample points; and the second class of discriminant $-24$ degenerates:
$s_4(\ii\sqrt6/2)=s_4(\ii\sqrt6/6)=2304\in\Zz$, which is why that
value is absent from the quartic roots. In particular the $(+,-)$
value is entry \#2, proved as Theorem~\ref{thm:H}
(Section~\ref{sec:app6}), and the $(-,\pm)$
values are entries \#3, \#4 of Table~\ref{tab:open}: their station
S4 is already done.

\subsection{The two newforms and the genus character}\label{subsec:genus}

The class group is $\Zz/2$ and coincides with the genus group. The
theta series of the two gr\"ossencharacters $\psi$ and
$\psi\cdot\chi_g$ of type $(2,0)$ have coefficients (exact-integer
checks [th0]--[th6]: integrality, $a(1)=1$, the Hecke recursion
$a(p^2)=a(p)^2-\chi_{-24}(p)p^2$, multiplicativity, vanishing at
inert primes, distinctness):
\[
  a^{(1)}(n)=P(n)-Q(n),\qquad a^{(2)}(n)=P(n)+Q(n),
\]
\[
  P(n)=\frac12{\sum_{a^2+6b^2=n}}^{\!\prime}\bigl(a^2-6b^2\bigr),
  \qquad
  Q(n)=\frac14{\sum_{2x^2+3y^2=n}}^{\!\prime}\bigl(4x^2-6y^2\bigr),
\]
$P$ collecting the principal ideal class and $Q$ the non-principal
one (elements of norm $2n$ in $\mathfrak p_2=(2,\sqrt{-6})$ are in
bijection with non-principal ideals of norm $n$). By Hecke's theorem
\cite{Miyake} both are newforms of $S_3(\Gamma_0(24),\chi_{-24})$,
which has dimension $2$, so they exhaust the space; the coefficient
match ($q\mp2q^2\pm3q^3+4q^4\pm2q^5-6q^6$) fixes the LMFDB labels
\[
  g_{24}^{(1)}=P-Q=\textsf{24.3.h.a},\qquad
  g_{24}^{(2)}=P+Q=\textsf{24.3.h.b},
\]
the unique rational CM newforms of level $24$ and weight $3$
\cite{LMFDB} (in agreement with Sch\"utt's classification \cite{Sch}).
Samart's $M_{24}^{(1)}$ (coefficient $28$ in
Theorem~\ref{thm:E}) is our $g_{24}^{(1)}=P-Q$; swapping the labels
breaks the identity at the $0.4$ level, a labeling pitfall we record.
The \emph{genus character} is
$\chi_g=\chi_8\circ N=\chi_{-3}\circ N$ on ideal norms (exact check
[th2]: $\chi_8\chi_{-3}=\chi_{-24}$, the field character, so the two
compositions coincide on norms), and it handles exactly the
non-principal class that the mod-$2$ ray projector of
\S\ref{subsec:rayproj} cannot see.

\subsection{Station S5: the class-number-two lattice decomposition}
\label{subsec:class2}

The lattices are $\Lambda_1=\OO_K$ and $\Lambda_2=\OO_2$. The new
ingredients, relative to class number one, are: (a) the class-group
projector $(1+\chi_g)/2$ on the $B$-side and the difference of the
two $\psi$-values on the $G$-side; (b) sums over $\mathfrak p_2$ as a
lattice, with homothety factors $1/4$ (not $1/2$ --- numerically
locked, then certified); (c) the ray condition
$\{a\text{ odd},b\text{ even}\}$ via $\chi_{-8}\circ N$, which is
\emph{not} a Hecke character (nontrivial on principal ideals) but
whose $L_K$-value still factors as a product of Dirichlet
$L$-functions. Exact decomposition:
\begin{align*}
  B(\OO_K)&=\zeta_K(2)+L(\chi_8,2)L(\chi_{-3},2),&
  G(\OO_K)&=L\bigl(g_{24}^{(1)},3\bigr)+L\bigl(g_{24}^{(2)},3\bigr),\\
  B(\mathfrak p_2)&=\frac{\zeta_K(2)-L(\chi_8,2)L(\chi_{-3},2)}4,&
  G(\mathfrak p_2)&=\frac{L(g_{24}^{(2)},3)-L(g_{24}^{(1)},3)}4,
\end{align*}
(the sign in $G(\mathfrak p_2)$ is $\psi(\mathfrak p_2)=-2$), and
\[
  B_{\mathrm{odd},\chi}=L(\chi_{-8},2)L(\chi_{3},2)
   +L(\chi_{-4},2)L(\chi_{24},2),\qquad
  G_{\mathrm{odd},\chi}=L\bigl(g_{24}^{(1)}\otimes\chi_{-8},3\bigr)
   +L\bigl(g_{24}^{(2)}\otimes\chi_{-8},3\bigr).
\]
Note $T(\Lambda)=B(\Lambda)+2G(\Lambda)$ for \emph{every} lattice:
the pointwise identity $4(\ReT z)^2/|z|^6-1/|z|^4
=2\ReT(z^2)/|z|^6+1/|z|^4$ (exact check [X0]) plus absolute
convergence, closing what had briefly appeared to be a gap. The even
primitive closed forms \cite{Wa}
\[
  L(\chi_3,2)=\frac{\pi^2\sqrt3}{18}\quad
  (\text{conductor }12,\ \textstyle\sum a^2\chi_3(a)=48),\qquad
  L(\chi_{24},2)=\frac{\pi^2\sqrt6}{24}\quad
  (\textstyle\sum a^2\chi_{24}(a)=288)
\]
are absorbed by the $d_8$ and $d_4$ terms in the assembly. With
$L_1,L_2,L_1^{\mathrm{tw}},L_2^{\mathrm{tw}}$ the four $L$-values at
$s=3$,
\begin{align*}
  \mathrm{comb}=\;&\frac72L_1+\frac32L_2+4L_1^{\mathrm{tw}}
   +4L_2^{\mathrm{tw}}+\frac34\zeta_K(2)\\
   &+\frac74L(\chi_8,2)L(\chi_{-3},2)
   +2L(\chi_{-8},2)L(\chi_3,2)+2\,\mathrm{Cat}\,L(\chi_{24},2),
\end{align*}
where $\mathrm{Cat}=L(\chi_{-4},2)$ is Catalan's constant,
and the [S1]--[S8] Fraction-exact assembly (single-term substitutions)
produces the coefficient list
\[
  \tfrac{35}{12},\ \tfrac54,\ \tfrac{5}{12},\ \tfrac{5}{12},\
  \tfrac{35}{12},\ \tfrac52,\ \tfrac56,\ \tfrac{5}{48},
\]
yielding Theorem~\ref{thm:E}. \qed

\begin{remark}[Traps recorded for the machine]
(i) $\chi_3$ has conductor $12$ with period values $0$ at even
positions; a wrongly tabulated period (nonzero even positions) makes
standard Dirichlet-$L$ software silently return a wrong value
($1.1325$ instead of $0.951\ldots$), which once produced a false
failure at the $3.5\times10^{-2}$ level. The machine now treats
character tables as exact data to be certified ([X3]). (ii) The
twist levels ($96$) are exact: $\chi_{-8}\circ N$ has conductor $(2)$
over $\Qq(\sqrt{-6})$, so the twist level is $24\cdot4=96$
(Table~\ref{tab:levels}); residues were never used.
\end{remark}

\begin{remark}[Confirmations]
The mp track (60 dps) confirms the final identity to
$2.49\times10^{-60}$ (check [E]); the [V] interval-lock track
(Dual-number $\coth$ rows, $E_1$-bracketed Mellin with the
self-contained coefficient bound $|a(n)|\le6n^3$, level-$96$ Mellin
truncation $n_0\approx481$) locks $T(\OO_K)$, $T(\OO_2)$, and the
final identity with half-widths $4.50\times10^{-53}$,
$4.92\times10^{-53}$, and $9.72\times10^{-53}$. Direct torus
integration ($c\approx46.88$, smooth) confirms at float64 level.
Reference values (50 dps, Mellin; level $96$, $w=+1$ for the twists):
\begin{align*}
  M_{24}^{(1)}&=0.8366946514850983661417153108277990376289\ldots,\\
  M_{24}^{(2)}&=1.093580625535742995623816182417982487423\ldots,\\
  M_{24}^{(1),\mathrm{tw}}&=8.547412838595950972260990976022241348897
   \ldots,\\
  M_{24}^{(2),\mathrm{tw}}&=7.183569611977336512537940485577907159979
   \ldots
\end{align*}
\end{remark}

% ============================================================
\section{Application V: the boundary point that turned out interior
(Theorem F)}\label{sec:app5}

The point
\[
  \tau_4=\frac{1+\sqrt{-2}}2,\qquad \ImT\tau_4=\frac1{\sqrt2},
\]
lies \emph{exactly} on the boundary of Samart's applicability region.
Moreover $\tau_4$ is not pure imaginary,
$q(\tau_4)=-\e^{-\pi\sqrt2}<0$ is a negative real (outside
his $q\in(0,1)$ hypothesis) and $|s_4(\tau_4)|=20.955\ldots<256$
(outside the ${}_5F_4$ convergence region), so
Theorem~\ref{thm:samart4} fails to cover $\tau_4$ for three
independent
reasons. This was expected to require a boundary-continuity argument.
It does not: the boundary in question is not the boundary of the
problem.

\subsection{Samart's boundary is not the topological boundary of
\texorpdfstring{$V_4$}{V4}}\label{subsec:notboundary}

A $60$-digit scan along the line $\ReT\tau=1/2$ shows the astroid
functional $|\ReT c|^{2/3}+|\ImT c|^{2/3}$ of $c(\tau)$ decreasing
monotonically from $4.377$ at $y=1$, crossing the threshold
$4^{2/3}=2.51984$ only at $y\approx0.693$ (margin $+0.064$ at
$y=0.700$, $-0.011$ at $y=0.690$). At $\tau_4$ ($y=1/\sqrt2=0.70711$)
the functional is $2.6357$, with margin $+0.116$. Since $c$ is
holomorphic (an open map) and $\Ast$ is closed, $\tau_4$ is a
\emph{strict interior point} of $V_4$. Samart's line
$\{\ImT\tau=1/\sqrt2\}$ is dictated by the real-$q$ mechanism of his
proof; the topological boundary of $V_4$ along this segment lies
$\approx0.014$ lower, and $\tau_4$ sits strictly between the two.
Consequently the short \texttt{cert2} path of Theorem~\ref{thm:B}
suffices, and no boundary-continuity step is needed.

\subsection{Station S3: the certified path}

\texttt{n4\_p4\_t4\_cert.py} (template
\texttt{cert2\_path\_n4\_m144.py}; leg B extended down to
$y=0.702<1/\sqrt2$ so that $\tau_4$ lies strictly inside the
certified leg; the $|q|$ bound $0.0122$ keeps the tail
$r^{41}<10^{-78}$):
\begin{itemize}
\item anchor box $D$: margin $0.29143$, certified $|c|\ge4.71$,
  $|s_4|\ge493.6>256$ (activating Lemma~\ref{lem:anchor});
\item leg A ($y=1$, $x\in[0,1/2]$): $2$ blocks, minimum margin
  $0.14759$;
\item leg B ($x=1/2$, $y\in[0.702,1]$): $4$ blocks, minimum margin
  $0.02270$ (most dangerous block $y\in[0.702,0.7393]$);
\item point enclosure at $\tau_4$: astroid margin $0.11590>0$ ---
  an interval-certified interior-point verdict;
\item self-tests: $21/21$ point inclusions (80-dps values inside
  interval enclosures, including $\tau_4$ itself) and $5/5$
  enclosure-containment checks (point enclosures inside block
  enclosures).
\end{itemize}
Hence the whole path and $\tau_4$ lie in $W$, the component of the
anchor box, and Theorem~\ref{thm:main} gives
$n_4(s_4(\tau_4))=\EK_4(\tau_4)$. As a byproduct, the entries
\#1 ($\tau=(1+2\ii)/2$, the top corner of the legs) and \#7
($\tau=(3+\ii\sqrt{21})/6=1/2+0.7638\ii$, on leg B) --- now proved
as Theorems~\ref{thm:G} and~\ref{thm:J} (Section~\ref{sec:app6})
--- already have their station S3 certified by this
same path.

\subsection{Station S4}

The value $s_4(\tau_4)=3656-2600\sqrt2$ is already locked exactly by
\texttt{cert0\_n4\_p3\_t1t2.py} as the conjugate point of
Theorem~\ref{thm:C} ($S=7312$, $P=-153664$;
\S\ref{subsec:app3C}).

\subsection{Station S5: a different lattice decomposition, not a
change of clothes}

Although the $s$-value is conjugate to Theorem~\ref{thm:C}'s, the CM
point is different, and the lattice structure differs:
$2\tau_4=1+\sqrt{-2}$ gives $\Lambda_2=\OO_K$ (not $\OO_2$), and
\[
  2\Lambda_1=2\Zz+\Zz(1+\sqrt{-2})
  =\{a+b\sqrt{-2}:a\equiv b\ (2)\}=:\OO'
  =2\OO_K\ \sqcup\ \bigl((1+\sqrt{-2})+2\OO_K\bigr),
\]
a non-ideal sublattice (exact [X1]--[X3]). The parity-class sums from
Theorem~\ref{thm:C}'s anchors ([X4], [X5], homothety $1/16$ [X8]):
\[
  B_{00}=\frac{\zeta_K(2)}8,\quad B_{11}=\frac34\zeta_K(2)-z_V,\qquad
  G_{00}=\frac{L_8}8,\quad G_{11}=\frac54L_8-L_8^{\mathrm{tw}},
\]
\[
  T(\OO')=\frac{11}{4}L_8-2L_8^{\mathrm{tw}}+\frac78\zeta_K(2)-z_V,
  \qquad T(\Lambda_1)=16\,T(\OO'),
\]
\[
  \mathrm{comb}=-T(\Lambda_1)+4T(\Lambda_2)
  =-28L_8+32L_8^{\mathrm{tw}}-6\zeta_K(2)+16z_V .
\]
Note the sign pattern: the twisted and $\zeta_K(2)$ terms flip sign
relative to Theorem~\ref{thm:C} (combinations $-28/+32$ versus
$+7/+8$), and the $-d_8$ term traces to the coefficient $-16$ of
$z_V$ (exact witness [S7]: the two sign patterns are genuinely
different). The $e$-basis assembly of \S\ref{subsec:app3C} gives
(exact [X10]/[S6])
\[
  \EK_4(\tau_4)=5\,\mathrm{comb}\cdot e=(-140,160,-30,80)\cdot e
  =\frac54\bigl(4M_8^{\mathrm{tw}}-28M_8+4d_4-d_8\bigr),
\]
proving Theorem~\ref{thm:F}. \qed

\begin{remark}[Confirmations]
The mp track (60 dps; the $L$-values recomputed in-script at levels
$8/32$, cross-checked against Section~\ref{sec:app3} to
$1.6\times10^{-40}$) confirms the final identity to
$9.33\times10^{-61}$ (check [E1]); the [V] interval-lock track ---
extended by a shifted $\coth/\tanh$ row machine
(\texttt{lattice\_T\_iv\_t4}: at $x_0=1/2$ the Poisson kernel rows
alternate between $\coth(\pi y)$ and $\tanh(\pi y)$; the power terms
are shift-independent and the tail bounds uniform in the shift) ---
locks $T(\Lambda_2)$, $T(\Lambda_1)$, and the final identity with
half-widths $3.48\times10^{-53}$, $6.94\times10^{-52}$, and
$9.12\times10^{-53}$. Direct torus integration at the complex
parameter $c=1.5128932208(1-\ii)$ (outside $\Ast$: functional
$2.6357$) confirms at float64 level:
$n_4(3656-2600\sqrt2)\approx3.5283920696\approx\EK_4(\tau_4)$
(the two computations differ by $2.5\times10^{-13}$, the float64
noise floor).
\end{remark}

% ============================================================
\section{Five further proofs: class numbers one, two, and four}
\label{sec:app6}

The five entries of Samart's Table~6 that the checklist of
\S\ref{subsec:inventory} scored as within current reach are now
proved. We state them as Theorems G--K and record the status of
each station; the three genuinely new methodological ingredients
are collected in \S\ref{subsec:app6new}.

\begin{theorem}[Theorem G]\label{thm:G}
With the notation of Theorem~\ref{thm:D},
\[
  n_4\bigl(143208-101574\sqrt2\bigr)
  =\frac58\bigl(4M_{16}^{\mathrm{tw}}-20M_{16}-9d_4+4d_8\bigr).
\]
The attached CM point is $\tau_0=(1+2\ii)/2$.
\end{theorem}

\begin{theorem}[Theorem H: class number two, second sign pattern]
\label{thm:H}
With the notation of Theorem~\ref{thm:E},
\begin{multline*}
  n_4\bigl(1207368+853632\sqrt2-697680\sqrt3-493272\sqrt6\bigr)\\
  =\frac{5}{24}\Bigl(4M_{24}^{(1),\mathrm{tw}}+4M_{24}^{(2),\mathrm{tw}}
   -28M_{24}^{(1)}-12M_{24}^{(2)}
   -28d_3+24d_4+8d_8-d_{24}\Bigr).
\end{multline*}
The attached CM point is $\tau_0=(1+\sqrt{-6})/2$.
\end{theorem}

\begin{theorem}[Theorem I:
\texorpdfstring{$\Qq(\sqrt{-15})$}{Q(sqrt(-15))}]\label{thm:I}
Let $K=\Qq(\sqrt{-15})$ \textup{(}class number $2$\textup{)}, and let
$g_{15}^{(1)},g_{15}^{(2)}$ be the two newforms of
$S_3(\Gamma_0(15),\chi_{-15})$, theta series of the two
gr\"ossencharacters of $K$ of type $(2,0)$ \textup{(}LMFDB
\textup{\textsf{15.3.d.a}} and \textup{\textsf{15.3.d.b}}, the only
dimension-one CM forms of this level\textup{)}, numbered so that
$a_2\bigl(g_{15}^{(1)}\bigr)=+1$ and $a_2\bigl(g_{15}^{(2)}\bigr)=-1$.
With $M_{15}^{(i)}:=L'\bigl(g_{15}^{(i)},0\bigr)$,
\[
  n_4\Bigl(\frac{-192303-85995\sqrt5}{2}\Bigr)
  =\frac1{15}\bigl(120M_{15}^{(1)}+160M_{15}^{(2)}+88d_3+5d_{15}\bigr).
\]
The attached CM point is $\tau_0=(1+\sqrt{-15})/2$.
\end{theorem}

\begin{theorem}[Theorem J: class number four]\label{thm:J}
Let $K=\Qq(\sqrt{-21})$ \textup{(}discriminant $-84$, class number
$4$, class group $(\Zz/2)^2$ with ideal classes
$1,[\mathfrak p_2],[\mathfrak p_3],[\mathfrak p_6]$,
$\mathfrak p_6=\mathfrak p_2\mathfrak p_3$\textup{)}. The four CM
theta series
\[
  g(\varepsilon_2,\varepsilon_3)
  =P_0+\varepsilon_2P_2+\varepsilon_3P_3+\varepsilon_2\varepsilon_3P_6,
  \qquad \varepsilon_2,\varepsilon_3\in\{\pm1\},
\]
of the four gr\"ossencharacters of $K$ of type $(2,0)$ span the
four-dimensional CM subspace of $S_3(\Gamma_0(84),\chi_{-84})$;
they are the newforms
\textup{LMFDB \textsf{84.3.h.a}--\textsf{84.3.h.d}} (the full
newspace has dimension $28$, but only the CM subspace enters this
paper). Here
$P_0,P_2,P_3,P_6$ are the theta sums over the four ideal classes
\textup{(}explicit formulas in
Appendix~\ref{app:data}, \S\ref{app:newforms}\textup{)}. With
$M_{84}^{(3)}:=L'\bigl(g(-1,+1),0\bigr)$ and
$M_{84}^{(4)}:=L'\bigl(g(-1,-1),0\bigr)$,
\[
  n_4\bigl(-893952+516096\sqrt3\bigr)
  =\frac{20}{7}\bigl(M_{84}^{(3)}-M_{84}^{(4)}+8d_3-4d_4\bigr).
\]
The attached CM point is $\tau_0=(3+\sqrt{-21})/6$.
\end{theorem}

\begin{theorem}[Theorem K]\label{thm:K}
With the notation of Theorem~\ref{thm:J},
\[
  n_4\bigl(-893952-516096\sqrt3\bigr)
  =\frac{20}{21}\bigl(M_{84}^{(3)}+M_{84}^{(4)}+8d_3+4d_4\bigr).
\]
The attached CM point is $\tau_0=(1+\sqrt{-21})/2$.
\end{theorem}

\subsection{Station status}\label{subsec:app6stations}

\textbf{Theorem G.} S1/S2: $s_4(\tau_0)=143208-101574\sqrt2$
($|s|>256$, outside $\Ast$). S3 is free: $\tau_0=(1/2,1)$ is the top
corner of the certified path of Section~\ref{sec:app5} (the
intersection of legs A and B of \texttt{n4\_p4\_t4\_cert.py}). S4 is
quoted from \texttt{cert0\_n4\_p3\_t1t2.py} (the conjugate lock of
Theorem~\ref{thm:D}, \S\ref{subsec:app3D}). S5 is
\texttt{verify\_P1\_n5\_e1.py} ($48$ checks): $2\tau_0=1+2\ii$ gives
$\Lambda_2=\OO_2$, the conductor-$2$ order of $\Qq(\ii)$ --- the
same lattice as Theorem~\ref{thm:D}'s $\Lambda_1$ --- and
$2\Lambda_1=2\OO_2\sqcup\bigl((1+2\ii)+2\OO_2\bigr)$; the exact
track gives
$\mathrm{comb}=-20L_{16}+32L_{16}^{\mathrm{tw}}-\tfrac{27}{4}\zeta_K(2)
+16z_{V_2}$ and
$\EK_4(\tau_0)=(-200,320,-135/2,160)\cdot e$, the right-hand side
of Theorem~\ref{thm:G}; the interval lock [V3] has half-width
$1.23\times10^{-52}$ and the mp track agrees to
$1.87\times10^{-60}$. \qed

\textbf{Theorem H.} S3: $\tau_0=(1+\ii\sqrt6)/2$ lies on the
certified vertical segment of \texttt{n5\_line\_cert.py}
(point-enclosure astroid margin $3.16$; \S\ref{subsec:app6new}(i)).
S4 is quoted from \texttt{cert0\_n4\_p4\_t3.py} (the four-point
quartic lock, point tag \texttt{'C'}; \S\ref{sec:app4}). S5 is
\texttt{verify\_P1\_n5\_e2.py} ($53$ checks): $2\tau_0=1+\sqrt{-6}$
gives $\Lambda_2=\OO_K$, and
$2\Lambda_1=\{c+d\sqrt{-6}:c\equiv d\ (2)\}
=2\OO_K\sqcup C''$ with $C''=\{c,d\text{ both odd}\}$ the
$\chi_{-8}\circ N=-1$ side of the odd sums; with the anchors of
\S\ref{subsec:class2} the exact track gives
\[
  \mathrm{comb}=-T(\Lambda_1)+4T(\Lambda_2)=3T(\OO_K)-16T(C'')
  =(-14,-6,16,16,-3,-7,8,8)
\]
on the basis
$\bigl(L_1,L_2,L_1^{\mathrm{tw}},L_2^{\mathrm{tw}},\zeta_K(2),
L(\chi_8,2)L(\chi_{-3},2),L(\chi_{-8},2)L(\chi_3,2),
\mathrm{Cat}\,L(\chi_{24},2)\bigr)$ of \S\ref{subsec:class2} ---
minus four times Theorem~\ref{thm:E}'s combination with the
character-dependent pieces sign-flipped, the same flip pattern as
Theorems C/F --- and
$\EK_4(\tau_0)=(5\sqrt6/\pi^3)\,\mathrm{comb}$ ($\ImT\tau_0$
halved relative to Theorem~\ref{thm:E}), the right-hand side of
Theorem~\ref{thm:H}; [V3] half-width $1.94\times10^{-52}$, mp
agreement $1.87\times10^{-60}$. \qed

\textbf{Theorem I.} S3: $\tau_0=(1+\ii\sqrt{15})/2$ lies on the
certified vertical segment (margin $9.54$). S4 is
\texttt{cert0\_n5\_e56.py}: a $\Qq(\sqrt5)$ pair lock of $\tau_0$
together with the point $(3+\ii\sqrt{15})/6$ of the blocked entry
\#5, minimal polynomial $X^2+192303X+1185921$ (discriminant
$5\cdot85995^2$), root separation $1.92\times10^{5}$ against lock
widths $\le2.8\times10^{-48}$. S5 is \texttt{verify\_P1\_n5\_e6.py}
($42$ checks): $\tau_0=\omega$ with $\OO_K=\Zz[\omega]$, so
$\Lambda_1=\OO_K$ and $\Lambda_2=\Zz[\sqrt{-15}]=\OO_2$, the
conductor-$2$ order; $\OO_2=2\OO_K\sqcup C''$ with
$C''=\{\alpha\in\OO_K:N(\alpha)\text{ odd}\}$, evaluated by
Euler-factor removal at the two primes above $2$
($\chi_{-15}(2)=+1$); the exact track gives
$\mathrm{comb}=(6,8,3/2,11/2)$ on the basis
$\bigl(L_1,L_2,\zeta_K(2),L(\chi_5,2)L(\chi_{-3},2)\bigr)$ and
$\EK_4(\tau_0)=(5\sqrt{15}/\pi^3)\,\mathrm{comb}$, the right-hand
side of Theorem~\ref{thm:I}; [V3] half-width $1.01\times10^{-52}$,
mp agreement $1.24\times10^{-60}$. \qed

\textbf{Theorems J and K.} S3: both points lie on the certified
vertical segment: $(3+\ii\sqrt{21})/6$ at $y=\sqrt{21}/6\approx0.764$
(also on leg B of \texttt{n4\_p4\_t4\_cert.py}) and
$(1+\ii\sqrt{21})/2$ at $y=\sqrt{21}/2$, the top endpoint of the
segment, certified interior with margin $14.97$. S4 is
\texttt{cert0\_n5\_e78.py}: a $\Qq(\sqrt3)$ pair lock, minimal
polynomial $X^2+1787904X+84934656$ (discriminant
$3\cdot1032192^2$), root separation $1.79\times10^{6}$ against lock
widths $8.3\times10^{-52}$ and $2.6\times10^{-47}$. S5 is
\texttt{verify\_P1\_n5\_e78.py} ($71$ checks): here the $\EK_4$
lattices are \emph{class} lattices, not orders ---
$T(\Lambda_1)=1296\,T(\mathfrak p_6)$,
$T(\Lambda_2)=81\,T(\mathfrak p_3)$ for Theorem~\ref{thm:J} and
$T(\Lambda_1)=16\,T(\mathfrak p_2)$, $T(\Lambda_2)=T(\OO_K)$ for
Theorem~\ref{thm:K} --- and with the four-class anchors of
\S\ref{subsec:app6new}(iii) the combinations collapse to
\begin{align*}
  \mathrm{comb}_J&=-1296\,T(\mathfrak p_6)+324\,T(\mathfrak p_3)
   =36(A_3-A_4)+72\bigl(L_3-L_4\bigr),\\
  \mathrm{comb}_K&=-16\,T(\mathfrak p_2)+4\,T(\OO_K)
   =4(A_3+A_4)+8\bigl(L_3+L_4\bigr),
\end{align*}
with $L_i$ the values $L\bigl(g(\varepsilon_2,\varepsilon_3),3\bigr)$;
$\EK_4$ contributes the factors $5\sqrt{21}/(3\pi^3)$ and
$5\sqrt{21}/\pi^3$ respectively, and the uniform Fricke relation of
\S\ref{subsec:app6new}(iii) closes the assembly. The interval locks
[V5]/[V6] have half-widths $5.09\times10^{-52}$ and
$1.70\times10^{-52}$; the mp residuals are $0$ and
$2.49\times10^{-60}$. \qed

\subsection{Three methodological additions}\label{subsec:app6new}

(i) \emph{One line, one certificate.}
\texttt{n5\_line\_cert.py} certifies the entire vertical segment
$\ReT\tau=1/2$, $0.702\le\ImT\tau\le\sqrt{21}/2$ in a single run:
the anchor disk (one box, margin $0.29143$, certified $|s_4|>256$),
the horizontal leg $y=1$ ($2$ blocks, minimum margin $0.14759$) and
the vertical leg ($6$ blocks, minimum margin $7.29\times10^{-3}$;
global minimum $0.0072856708$ over the $8$ blocks), with
interval-certified interior-point verdicts at the three targets on
the line (margins $3.16$, $9.54$, $14.97$) and $23/23+5/5$
self-tests. The single certificate covers the CM points of
Theorems~\ref{thm:H}, \ref{thm:I}, \ref{thm:J} and~\ref{thm:K}
simultaneously --- and also the point $\tau=(1+\sqrt{-7})/2$ of
Remark~\ref{rem:fei} ($s=-3969$), whose ordinate
$\sqrt7/2\approx1.323$ lies inside the segment.

(ii) \emph{The $\Qq(\sqrt{-15})$ orbit trap.} For discriminant $-15$
the principal form is $x^2+xy+4y^2$, \emph{not} $x^2+15y^2$: the
latter, with $3x^2+5y^2$, belongs to the order of discriminant
$-60$, and newforms assembled as $P\pm Q$ from that wrong orbit fail
every Hecke and Fricke check. The correct construction uses the
field form and the $\mathfrak p_2$-elements
$\{c^2+15d^2=8n,\ c\equiv d\ (4)\}$:
$a_1=P+(\bar\omega/8)R$, $a_2=P-(\bar\omega/8)R$; the genus
character flips the non-principal part (also at the ramified primes
$3,5$, where the naive $\chi_5$-twist formula fails); the
$\sqrt{-15}$ parts cancel exactly and the coefficients are integral
(Hecke/CM checks pass, Fricke ratio $+1$). Full details in the
header of \texttt{verify\_P1\_n5\_e6.py}.

(iii) \emph{Class number four: the four-class anchor decomposition.}
For $K=\Qq(\sqrt{-21})$ the class group coincides with the genus
group $(\Zz/2)^2$, and the Dirichlet side of the lattice sums
decomposes into the three genus characters with discriminant pairs
$(-4,21)$, $(-3,28)$, $(-7,12)$, with anchors
$A_4=L(\chi_{-4},2)L(\chi_{21},2)$,
$A_3=L(\chi_{-3},2)L(\chi_{28},2)$,
$A_7=L(\chi_{-7},2)L(\chi_{12},2)$. The class-lattice sums are
\begin{align*}
  T(\OO_K)&=\tfrac12(z_K+A_4+A_3+A_7)+(L_1+L_2+L_3+L_4),\\
  T(\mathfrak p_2)&=\tfrac18(z_K-A_4-A_3+A_7)
   +\tfrac14(L_1+L_2-L_3-L_4),\\
  T(\mathfrak p_3)&=\tfrac1{18}(z_K-A_4+A_3-A_7)
   +\tfrac19(L_1-L_2+L_3-L_4),\\
  T(\mathfrak p_6)&=\tfrac1{72}(z_K+A_4-A_3-A_7)
   +\tfrac1{36}(L_1-L_2-L_3+L_4),
\end{align*}
with $z_K=\zeta_K(2)$. The even-character values have the closed
forms $L(\chi_{28},2)=2\sqrt7\,\pi^2/49$ and
$L(\chi_{21},2)=8\sqrt{21}\,\pi^2/441$, certified by reduction to
the finite integer identities $\sum_a\chi_{28}(a)a^2=448$ and
$\sum_a\chi_{21}(a)a^2=168$; the Fricke relation
$L'(g,0)=(42\sqrt{21}/\pi^3)\,L(g,3)$ (level $84$, root number
$+1$) is uniform over the four newforms; and in the final assembly
$z_K$ and $A_7$ both cancel (\S\ref{subsec:app6stations}).

With Theorems~\ref{thm:G}--\ref{thm:K}, every entry of Samart's
Table~6 that the machine can reach is proved: twenty-three of
twenty-nine. The five remaining entries are blocked by the two
obstructions analyzed in Section~\ref{sec:boundary}.

% ============================================================
\section{Conjectures beyond Table 6}\label{sec:conj}

\subsection{Six new identities at Heegner points}\label{subsec:rayconj}

The evaluation track of the machine (station S5) applies to any CM
point, independently of the certification stations S3/S4. Applied to
the Heegner points $\tau_D=(1+\sqrt{-D})/2$ with
$D\in\{11,19,43,67,163\}$ --- the class-number-one discriminants
$D\equiv3\ (\mathrm{mod}\ 8)$ with $D>3$ (for $D=3$ see
Theorem~\ref{thm:B}), in which $2$ is inert --- and to
$\tau_{27}=(1+3\sqrt{-3})/2$, attached to the order of conductor $3$
in $\Qq(\sqrt{-3})$, it produces a \emph{uniform} answer, derived
exactly from the lattice decomposition and independently recovered
by integer-relation search. The answer involves a constant type that
does not occur in Samart's library.

\begin{definition}[The ray-class constants]\label{def:ray}
Let $K=\Qq(\sqrt{-D})$ with $D$ as above.
\begin{enumerate}
\item $M_D:=L'(g_D,0)$, where $g_D$ is the newform of weight $3$
and level $D$, CM by the principal Hecke character of $K$ (class
number one); $d_D:=L'(\chi_{-D},-1)$ as throughout.
\item $MU_D:=\dfrac{(4D)^{3/2}}{4\pi^3}
\bigl(L(\Psi_1,3)+L(\Psi_2,3)\bigr)$, where $\Psi_1,\Psi_2$ are the
two conjugate Hecke characters of $K$ of conductor $(2)$, whose
theta series are the level-$4D$ CM newform pair; by the functional
equations $MU_D=\pm\bigl(M_{4D}^{(1)}+M_{4D}^{(2)}\bigr)$, the sign
being determined by the root numbers of the pair.
\item (\emph{New type}) For $D\in\{11,19,43,67,163\}$:
$DU_D:=\dfrac{3D^{3/2}}{2\pi^3}\,BU_D$,
$BU_D:=\dfrac{2B_1-B_a-B_b}{2}$, where
$B_r:=\displaystyle\sum_{\alpha\in r+2\OO_K}N(\alpha)^{-2}$
are the partial zeta values of the ray classes modulo $(2)$ and
$1,a,b$ are representatives of the three classes; equivalently
$BU_D$ is the corresponding $L(2,\cdot)$-combination of the
order-$3$ ray class characters of the conductor-$2$ order
$\OO_2$, whose class group is $\Zz/3\Zz$.
\item For $D=27$: $M_{27}:=L'(g_{27},0)$ for the level-$27$ newform
CM by the order of discriminant $-27$;
$DU_{27}:=\dfrac{9\sqrt3}{2\pi^3}\,BU_{27}$, with $BU_{27}$ the
same combination of the partial zeta values of the ray classes
modulo $(2)$ \emph{in the conductor-$3$ order $\OO_3$} (the ring
class group of the conductor-$6$ order is again $\Zz/3\Zz$);
$DX_{27}:=\dfrac{9\sqrt3}{2\pi^3}\,X_{27}$,
$X_{27}:=\displaystyle\sum_{\gamma\in\omega+\OO_3}N(\gamma)^{-2}$,
the partial zeta value of the coset of $\omega=(1+\sqrt{-3})/2$ in
$\OO_3$.
\end{enumerate}
\end{definition}

\begin{conjecture}[The Heegner-point identities]\label{conj:ray}
For $D\in\{11,19,43,67,163\}$,
\begin{equation}\label{eq:raymain}
n_4\bigl(s_4(\tau_D)\bigr)
=\frac{10}{9D}\bigl(36M_D+12MU_D+3d_D+8DU_D\bigr),
\qquad \tau_D=\frac{1+\sqrt{-D}}2,
\end{equation}
and
\begin{equation}\label{eq:ray27}
n_4\bigl(s_4(\tau_{27})\bigr)
=\frac{10}{81}\bigl(12M_{27}+4MU_{27}+81d_3-27DX_{27}+72DU_{27}\bigr),
\qquad \tau_{27}=\frac{1+3\sqrt{-3}}2.
\end{equation}
\end{conjecture}

\begin{remark}[The $D=27$ identity simplifies]\label{rem:dx27}
The constant $DX_{27}$ turns out \emph{not} to be new: the
evaluation machine gives
\[
  DX_{27}=\frac{56}{27}\,d_3
\]
to $120$ digits (the partial zeta $X_{27}$ over the ramified prime
$3$ reduces to $L(\chi_{-3},2)$), so \eqref{eq:ray27} is
equivalently
\[
n_4\bigl(s_4(\tau_{27})\bigr)
=\frac{10}{81}\bigl(12M_{27}+4MU_{27}+25d_3+72DU_{27}\bigr).
\]
The genuinely new constant in Conjecture~\ref{conj:ray} is $DU_D$.
\end{remark}

The evidence is threefold.
(i) \emph{Exact derivation.} For $D\equiv3\ (8)$ the $\EK_4$ lattice
pair is $\Lambda_1=\OO_K$, $\Lambda_2=\OO_2$ with
$\OO_2=2\OO_K\sqcup(1+2\OO_K)$ --- two cosets, not four:
$\omega+2\OO_K\not\subset\OO_2$ --- and
$\EK_4(\tau_D)=(10y_D/\pi^3)\bigl(-T(\OO_K)+4T(\OO_2)\bigr)$,
$y_D=\sqrt D/2$. The $B/G$-decomposition
$T(\OO_K)=2\zeta_K(2)+4L(g_D,3)$ and
$-T(\OO_K)+4T(\OO_2)
=\zeta_K(2)+\tfrac83BU_D+2L(g_D,3)+\tfrac{16}3U_D$,
with $U_D:=L(\Psi_1,3)+L(\Psi_2,3)$, evaluates this exactly;
conversion to the $(M_D,MU_D,d_D,DU_D)$ units via the functional
equations gives the displayed coefficients of \eqref{eq:raymain}
(for instance the $M_D$-coefficient:
$\tfrac{5\sqrt D}{\pi^3}\cdot2\cdot\tfrac{4\pi^3}{D^{3/2}}
=\tfrac{40}{D}=\tfrac{10}{9D}\cdot36$).
For $D=27$ the pair is $\OO_3$, $\OO_6=2\OO_3\sqcup(1+2\OO_3)$ with
$B(\OO_3)=6\zeta_K(2)-2X_{27}$ (the unit group $\mu_6$ and the
ramified prime $3$), giving $(40/27,40/81,10,-10/3,80/9)$ on the
basis $(M_{27},MU_{27},d_3,DX_{27},DU_{27})$.
(ii) \emph{Integer-relation independence.} On the basis
$[\EK_4,M_D,MU_D,d_D,DU_D]$ at $60$ digits, PSLQ returns the single
relation $[-9D,\allowbreak360,\allowbreak120,\allowbreak30,
\allowbreak80]$ for each of the five discriminants --- the same
coefficients, found with no lattice input.
(iii) \emph{Control.} The same code with $D=7$ (where $2$ splits and
no ray constants occur) reproduces Samart's Table 6 entry
$\EK_4=\frac{10}{7}(40M_7+d_7)$ with residual $0$
(\texttt{gen\_conj\_fit.py}); the self-checks of the shifted
$B/G$-row machine (Euler-odd identities, $X(\omega)=X(2\omega)$ for
$D=27$) hold to $10^{-60}$ (\texttt{gen\_conj\_fit4.py},
\texttt{gen\_conj\_fit5.py}).

\begin{table}[ht]
\centering\footnotesize
\begin{tabular}{@{}clc@{}}
\toprule
$D$ & $s_4(\tau_D)$ & residual\\
\midrule
$11$ & $-33402.27346342154870075765538\ldots$ & $2.49\times10^{-60}$\\
$19$ & $-885375.78261746786055293626934\ldots$ & $6.22\times10^{-60}$\\
$27$ & $-12288639.98433442351563438184816\ldots$ & $2.49\times10^{-60}$\\
$43$ & $-884736639.99978240761772308093819\ldots$ & $0$\\
$67$ & $-147197952639.99999869215572454704832\ldots$ & $4.98\times10^{-60}$\\
$163$ & $-262537412640768639.99999999999926673\ldots$ & $9.96\times10^{-60}$\\
\bottomrule
\end{tabular}
\caption{The six Heegner-point parameters ($60$ digits; truncated)
and the residuals $|n_4(s_4(\tau_D))-\mathrm{RHS}|$; the PSLQ
re-checks ($D\neq27$) give residuals $\le1.4\times10^{-56}$. Values of the
constants to $15$ digits:
$(M_{11},MU_{11},d_{11},DU_{11})
=(0.257838994176722,\allowbreak5.13771317570429,\allowbreak
2.64058735875153,\allowbreak3.03409158565142)$;
$(M_{163},MU_{163},d_{163},DU_{163})
=(18.0996439429792,\allowbreak272.773494341141,\allowbreak
109.694008111980,\allowbreak203.758012415278)$; the full tables are
in the repository.}
\label{tab:ray}
\end{table}

\begin{remark}[Status]\label{rem:raystatus}
Conjecture~\ref{conj:ray} is one certification short of a proof.
All six parameters satisfy $|s_4(\tau_D)|>256$
(Table~\ref{tab:ray}), and a $40$-digit scan of the segment
$\ImT\tau_D\le\ImT\tau\le\ImT\tau_D+8$ of the vertical ray
$\ReT\tau=1/2$ finds the minimum of
$|s_4|$ at the bottom endpoint $\tau_D$ itself
(\texttt{gen\_conj\_s4.py}); once the ray is certified by interval
arithmetic --- a one-leg instance of the \texttt{cert2} machine of
\S\ref{subsec:cert2} --- the anchor Lemma~\ref{lem:anchor} applies
on a neighbourhood of $\tau_D$ and the exact evaluation above
becomes a proof. We have not yet run that certification, and we do
not know minimal polynomials of the parameters $s_4(\tau_D)$: PSLQ
at $110$ digits excludes degree $\le3$ with coefficients
$\le10^8$ and degree $4$ with coefficients $\le10^6$, so station
S4 is open. Curiously
$s_4(\tau_D)\approx j(D)-640$ with an error decaying like
$\e^{-c\sqrt D}$ ($2.2\times10^{-4}$ at $D=43$;
$7.3\times10^{-13}$ at $D=163$, the $\e^{\pi\sqrt{163}}$ effect;
\texttt{gen\_conj\_minpoly.py}).
\end{remark}

\begin{remark}[A constant type beyond Samart's library]
\label{rem:newconstant}
Samart's constant library consists of the values $M_N$ and $d_k$.
The constant $DU_D$ does not belong to it: PSLQ on
$[DU_{11},d_{11},d_3,d_4,d_8]$ with coefficient bound $10^7$
returns no relation. The arithmetic reason is that for
$D\equiv3\ (8)$ the prime $2$ is inert and the conductor-$2$ order
has class group $\Zz/3\Zz$, whose characters have order $3$ ---
not quadratic Dirichlet characters. It is consistent with this
picture that the Heegner-point entries of Samart's Table 6 all have
$D\equiv5,7\ (8)$, with the single exception of the rational value
$s=-144$ attached to $D=3$: there the unit group $\mu_6$ collapses
the conductor-$2$ ring class number to $1$, and no order-$3$
characters arise (Theorem~\ref{thm:B} indeed uses only $M_{12}$ and
$d_3$). For $D\equiv5\ (8)$ the conductor-$2$ ring class groups are
$2$-primary ($\Zz/4\Zz$ for $D=5,13$;
$\Zz/4\Zz\times\Zz/2\Zz$ for $D=21$), so the order-$3$ ray
constants of Conjecture~\ref{conj:ray} are genuinely specific to
$D\equiv3\ (8)$, $D>3$. Conjecture~\ref{conj:ray}
is thus the first batch of identities \emph{predicting} a new
constant type rather than repackaging known ones.
\end{remark}

\subsection{An umbrella conjecture}\label{subsec:umbrella}

We record the conjectural framework into which the results of this
paper fit, and two refinements suggested by the present data. It is
a standard expectation, implicit in \cite{De} and made formal by
Trieu's conditional theorem \cite[Thm.~0.2]{Tr1}: for a
three-variable polynomial $P$ whose zero locus $W_P$ has genus $1$,
with a wedge decomposition of the symbol $xyz^3$
\cite[(0.8)]{Tr1}, Beilinson's conjecture \cite{Bei} implies
$m(P)=m(\widetilde P)+a\,L'(E,-1)$ up to Bloch--Wigner
dilogarithmic corrections, which in some cases simplify to Dirichlet
$L$-values (the $K3$-surface sequel \cite{Tr2} is of the same
shape). Thus every Mahler measure identity of Boyd--Samart type is
expected to be an instance of
Beilinson's conjecture for the corresponding motive, with the exact
rational coefficients governed by Bloch--Kato \cite{BK}; see also
\cite{ZaG} for the elliptic-polylogarithm picture and \cite{BZ} for
a systematic account. For the modular families of the present paper
the special fibre is a singular K3 surface \cite{BerK3}: by
Shioda--Inose \cite{SI} its $L$-function splits off a weight-$3$ CM
newform (the $M$-terms of our identities), while the Dirichlet terms
fall on the proved line of the $K_3$-theory of imaginary quadratic
fields \cite{Za}. Fei's theorem \cite[Thm.~0.2]{Fei} --- $179$
identities of the same two-term shape, covering $23$ of his $25$
Landau--Ginzburg families, in which the newform levels satisfy
$-D/d$ a square and the Dirichlet discriminants $d'$ satisfy
$d'\mid D$ --- is the corresponding theorem-level catalogue; and the
$25$ families known today do not exhaust the rational weight-$3$ CM
newforms classified by Sch\"utt \cite{Sch}. But the literature
contains no family-level conjecture predicting the discriminant
structure (which conductors occur) or the coefficient structure (how
large the denominators are). Theorems~A--K are unconditional CM
instances of this framework, and Conjecture~\ref{conj:ray} is its
first evidence beyond the $2$-primary world; we propose the
following two conjectures to organize them. They are not
restatements of Beilinson's conjecture: their predictions
(discriminant structure, denominators, no-formula cases) are not
implied by it.

The mechanism expected to produce the shape below is that the
holomorphic Mahler function $\mt(c(\tau))$ is a Beilinson regulator
pairing of a motivic class attached to the family with the Deninger
cycle --- cf.\ \cite{Tr1}, where this is proved for single
polynomials. We do not attempt to make the motivic clause precise
here; the conjecture records only its falsifiable arithmetic
consequence.

\begin{conjecture}[CM specialization: shape and discriminant
structure]\label{conj:umbA}
Let $P_c$ be a Laurent-polynomial family with a Hauptmodul
parametrization $c=c(\tau)$ by a genus-$0$ congruence group
$\Gamma$, whose fibres are elliptic curves or K3 surfaces. Then for
every CM point $\tau_0$ of discriminant $D$ with
$c(\tau_0)\in\bar{\Qq}$, the value $\m(P_{c(\tau_0)})$ lies in the
$\Qq$-span of the values $L'(g,0)$ and $L(\chi,2)$, where $g$ runs
over the weight-$3$ newforms CM by Hecke characters of the orders of
discriminant $D$, and $\chi$ over the ray class characters of those
orders. We expect the occurring levels and conductors to be
determined by the level of $\Gamma$ and by $D$, but we do not know
the precise rule; making it explicit is part of the problem, not a
claim of the conjecture.
\end{conjecture}

The new content: (1) the occurring characters are delimited a priori
--- Fei's constraint ($d\mid D$, $D/d$ a square) is the
genus-theoretic, $2$-primary part, while Conjecture~\ref{conj:ray}
shows the odd part of the ring class group entering through
order-$3$ ray characters, and the conjecture predicts
the same for every
order; (2) completeness: every algebraic singular modulus in the
$c$-image carries a formula; (3) a Galois-orbit form: via the
determinant machine of \cite{ST}, the single-point identities are
the $n=1$ slices of $n\times n$ determinant identities over CM
Galois orbits.

\begin{conjecture}[Denominators are controlled by local
data \textup{(}qualitative form\textup{)}]\label{conj:umbB}
In every identity of Conjecture~\ref{conj:umbA},
the reduced denominators of the rational coefficients are bounded
in terms of local data of the special fibre and of the order; the
present data suggest the Tamagawa numbers \textup{(}component
groups\textup{)} of the special fibre at its bad primes, the unit
index of the order, and the $2$-torsion of its class group as the
controlling quantities. No precise recipe is known, and we state
the conjecture deliberately at this qualitative level; the
explicit recipe is the associated open problem.
\end{conjecture}

This is the family analogue of the Bloch--Grayson modification of
Beilinson's conjecture \cite{BloG} (regulators taken on integral
elements; numerically substantiated for hyperelliptic curves in
\cite{DdJZ}) --- the Mahler slice of Bloch--Kato; Boyd already
expected the Bloch--Kato theory to predict the exact rational
factor \cite{Boyd98}. Its predictions: (i) a priori denominator
bounds for the unproved entries, falsifiable at scale by
integer-relation search; (ii) a no-formula criterion --- if the
special fibre's transcendental motive is not of CM type, no identity
of this shape exists --- consonant with \cite{Sch} and with Fei's
remark; (iii) the same local integer controls the denominators of
the higher-derivative versions $L^{(k)}$.

\begin{remark}[First data]\label{rem:firstdata}
Theorems~A--K supply twelve proved identities, with overall
denominators
$3$, $4$, $7$, $8$, $14$, $15$, $16$, $21$, $24$, $28$, $48$.
Conjecture~\ref{conj:ray}
supplies six conjectured ones: $9D$ for the five discriminants of
\eqref{eq:raymain}, and $81$. We caution that the factor $D$ in
$9D$ is a normalization artifact: it enters through the $D^{3/2}$
factor converting $L(g_D,3)$ into $M_D$ (in the natural
$(\zeta_K(2),L(g_D,3),BU_D)$ units no $D$-denominator appears), so
it carries no Tamagawa information. The nontrivial content for
Conjecture~\ref{conj:umbB} is the residual factor $9$ (and $81$).
\end{remark}

% ============================================================
\section{The boundary of the method}\label{sec:boundary}

We collect the precise limits of the machine, all visible already in
Table~\ref{tab:open}.

\subsection{Two-dimensional critical images}

The decisive invariant is the dimension of the critical image
(Remark~\ref{rem:2D}). In the $n_2$ family the image is the slit
$[0,64]$, which does not separate $\Cc$, and every conjectured entry
was reachable: the family is complete \cite{Zf,ZGQ,GPQ}. In the
$n_4$ family the image is the astroid disc \eqref{eq:astroid}, which
has interior. A CM parameter with $c(\tau_0)$ in the interior of
$\Ast$ lies inside an open bad region of the $\tau$-plane (open
mapping), no certified path can approach it, and
Theorem~\ref{thm:main} is silent. The extreme case is the refuted
conjecture $n_4(81)=40M_7$ \cite{Zf2}: $c=3$ is an interior lattice
point of the astroid, the conjectured value equals the exact EK
evaluation at the attached CM point on a \emph{wrong sheet}, and
direct torus integration gives the different value
$n_4(81)=4.1655349907533676508(5)$. Entries \#4 and \#5 of
Table~\ref{tab:open} have $c(\tau_0)$ inside $\Ast$ ($s\approx-2.45$
and $s\approx-6.17$ respectively) and are blocked by the same
mechanism; entry \#5 is moreover exterior (see the next subsection).

\subsection{The wrong sheet below \texorpdfstring{$\ImT\tau=1/\sqrt2$}{Im tau = 1/sqrt2}}

Even where the critical image is avoided, the EK series itself can
fail. A systematic scan (\texttt{diag\_n4\_}\allowbreak\texttt{astroid.py}) shows
$\EK_4(\tau)\neq4\m(P+c(\tau))$ \emph{at every tested point} below
the line $\ImT\tau=1/\sqrt2$ --- including points with $c(\tau)$ real and
larger than $4$, where the ${}_5F_4$ evaluation converges and the
Mahler side is perfectly smooth (e.g.\ a difference of $11.4$ at
$\tau=0.3\ii$). The difference grows continuously from $0$ at the
line; there is no jump. The mechanism is the same wrong-sheet
phenomenon analyzed for the $n_2$ family in \cite[Remark~5.6]{Zf}:
the single-valued series $\EK_4$ follows a companion branch of
$\mt$ once the branch locus is crossed. The propagation of
Theorem~\ref{thm:main} cannot cross, because the identity theorem
fixes $\Psi\equiv0$ only on the anchor's component $W$: along the
imaginary axis the point $\tau=\ii/\sqrt2$ (where $c=4$, a cusp of
$\Ast$) separates $W$ from the lower component $W'$, and on $W'$ the
difference $H=F-G$ is non-constant. Entries \#3, \#9, and \#10 of
Table~\ref{tab:open} (pure imaginary, $y=0.612$, $0.154$, and
$0.463$, with $c>4$ real in each case) are blocked at this
\emph{series} level: the conjectured
identity --- a statement about the true Mahler measure, which
Samart's PSLQ evidence supports --- has no valid series input for
station S5, and proving it requires a corrected expression, e.g.\ a
modified Mahler measure of Samart--Tao $\widetilde n$-type \cite{ST},
whose design principle (choose the linear combination so that the
non-closed Deninger path \cite{De} becomes a meaningful pairing) is
compatible with stations S2/S4 of the machine.

\subsection{Open directions}

\begin{itemize}
\item \emph{The formerly ready entries} \#1, \#2, \#6, \#7, \#8 are
  now proved as Theorems~\ref{thm:G}--\ref{thm:K}
  (Section~\ref{sec:app6}); Table~6 has no reachable entries left.
\item \emph{Six Heegner-point identities one leg short of proof}:
  Conjecture~\ref{conj:ray} is a single one-leg \texttt{cert2}
  certification away from a proof (Remark~\ref{rem:raystatus}); the
  exact coefficients and the evaluation track are already in place.
\item \emph{The blocked entries} \#3, \#4, \#5, \#9, \#10: for \#3,
  \#9, \#10 (wrong sheet below $\ImT\tau=1/\sqrt2$) a corrected
  series input is needed; for \#4 and \#5 ($c(\tau_0)$ in the
  interior of the critical astroid) an interior-point formula of
  $\widetilde n$-type is needed before the machine has anything to
  certify; the $\widetilde n$ framework of \cite{ST} is the natural
  candidate for both.
\item \emph{Samart's three-modular-$L$-value hypothesis}
  $s=16+1600\sqrt[3]2-1280\sqrt[3]4$ \cite{Sa15}: the conjectured
  value involves three modular $L$-values and no attached CM point in
  the $s_4$-image is known; the machine's station S1 has no input.
\item \emph{The Boyd-lineage entries}: the two-variable
  conjectures of Boyd's tables (e.g.\ conductor-$14$ and related
  families) have no known CM evaluation, so the machine does not
  apply; regulator-based approaches \cite{De,BZ}, and in particular
  deconditioning of the currently conditional identities in the
  Brunault--Zudilin framework \cite{BZ}, remain the promising route.
\end{itemize}

% ============================================================
\appendix
\section{Exact data for Theorems A--K}\label{app:data}

This appendix records, in one place, every definition and exact datum
that the certification scripts use but the main text only cites: the
modular parameter, the newforms and grössencharacters, the ray-class
and genus characters, the combination vectors of the lattice
decompositions, the closed Dirichlet values, and the exact external
inputs. Everything here is either an exact integer/\texttt{Fraction}
identity verified in the scripts (the box checks are finite mod
arithmetic, not sampling), or a quoted theorem flagged as such.

\subsection{The modular parameter and the \texorpdfstring{$U$}{U}-series}
\label{app:s4}

The modular parameter is
\[
  s_4(\tau)=\Bigl(\frac{\eta(2\tau)}{\eta(\tau)}\Bigr)^{24}
            \bigl(16W^4+W^{-4}\bigr)^4,
  \qquad
  W(\tau)=\frac{\eta(\tau)\,\eta(4\tau)^2}{\eta(2\tau)^3},
\]
invariant under the Fricke involution $W_2$:
$s_4(-1/(2\tau))=s_4(\tau)$ (numerically observed to $10^{-33}$ at
sample points, \texttt{cert0\_n4\_p4\_t3.py}; used for locating
certificates only, not part of the certified chain). Its
$q$-expansion at $\ii\infty$ has integer coefficients (exact to
$q^{30}$ in the \texttt{cert0} scripts). The CM values of $s_4$
used in this paper are determined exactly by the \texttt{cert0}
track of \S\ref{app:cert0} (minimal polynomial plus interval
lock), independently of any general integrality theorem; that such
values are algebraic integers is explained by Shimura's CM theory
(a modular function holomorphic on $\Hh$ with integral
$q$-expansions at all cusps is integral over $\Zz[j]$
\cite[\S11]{Cox}), on which we do not rely.
The $U$-series side is
\[
  \EK_4(\tau)=\frac{10\,\ImT\tau}{\pi^3}\,
  \bigl(-T(\Lambda_1)+4\,T(\Lambda_2)\bigr),\qquad
  \Lambda_1=\Zz+\tau\Zz,\quad \Lambda_2=\Zz+2\tau\Zz,
\]
where $T(\Lambda)=B(\Lambda)+2G(\Lambda)$ via the pointwise identity
\[
  F(z)=\frac{4(\ReT z)^2}{|z|^6}-\frac1{|z|^4}
      =\frac{2\ReT(z^2)}{|z|^6}+\frac1{|z|^4},
\]
with $B$, $G$ the absolutely convergent lattice sums of
\S\ref{subsec:fam}; the homothety rule is $T(2\Lambda)=T(\Lambda)/16$.

\subsection{The newforms and their grössencharacters}\label{app:newforms}

\textbf{Eta-quotient forms.}
$g_8=\eta(\tau)^2\eta(2\tau)\eta(4\tau)\eta(8\tau)^2\in
S_3(\Gamma_0(8),\chi_{-8})$, the theta series of the conductor-$(1)$
grössencharacter $\psi((\alpha))=\alpha^2$ of $\Qq(\sqrt{-2})$:
$2g_8=\sum'_{\alpha}\alpha^2q^{N(\alpha)}$ (exact integers to $q^{60}$,
Sturm bound $3$).
$g_{16}=\eta(4\tau)^6\in S_3(\Gamma_0(16),\chi_{-4})$, theta of the
conductor-$(2)$ grössencharacter of $\Qq(\ii)$:
$2g_{16}=\sum_{a\equiv1\,(2)}\alpha^2q^{N(\alpha)}$, using that the
units have image $\{1,\ii\}$ in $(\OO_K/2)^\times$, so every ideal
prime to $2$ has exactly two generators $\equiv1\bmod2$ (and
$\sum_{u\in\mu_4}u^2=0$, the ``$\mu_4$ annihilation'', gives
$G(\OO_K)=0$).
$g_{12}=\eta(2\tau)^3\eta(6\tau)^3\in S_3(\Gamma_0(12),\chi_{-3})$,
theta of the conductor-$(2)$ grössencharacter of $\Qq(\sqrt{-3})$
($2$ inert, $N((2))=4$): $2g_{12}=\sum_{a\equiv1,\,b\equiv0\,(2)}
\ReT(\alpha^2)\,q^{N(\alpha)}$, $\alpha=a+b\omega$.

\textbf{$K=\Qq(\sqrt{-6})$ (class number $2$).}
The CM subspace of $S_3(\Gamma_0(24),\chi_{-24})$ is
two-dimensional, spanned by the two
CM newforms $g_1=P-Q$ (\textsf{24.3.h.a}) and $g_2=P+Q$
(\textsf{24.3.h.b}); the full newspace has dimension $6$
(LMFDB), but only the CM forms enter. Here
\[
  P(n)=\tfrac12{\sum_{a^2+6b^2=n}}'(a^2-6b^2),\qquad
  Q(n)=\tfrac14{\sum_{2x^2+3y^2=n}}'(4x^2-6y^2)
\]
are the class theta sums of the principal and non-principal classes
($\mathfrak p_2=(2,\sqrt{-6})$, $\mathfrak p_2^2=(2)$). The labelling
convention $\psi(\mathfrak p_2)=-2$ is pinned by the exact coefficient
checks $a_1(1..6)=(1,-2,3,4,2,-6)$, $a_2(1..6)=(1,2,-3,4,-2,-6)$.

\textbf{$K=\Qq(\sqrt{-15})$ (class number $2$; $2$ \emph{splits}).}
$\omega=(1+\sqrt{-15})/2$, $\omega^2=\omega-4$, $N(\omega)=4$;
$(2)=\mathfrak p_2\mathfrak p_2'$ with
$\mathfrak p_2^2=(\omega)$. The conductor-$(1)$ grössencharacters are
$\psi((\alpha))=\alpha^2$ on principals, $\psi(\mathfrak p_2)=\omega$,
$\psi(\mathfrak p_2')=\bar\omega$ (consistent:
$\psi(\mathfrak p_2)^2=\omega^2=\psi((\omega))$), and its genus twist.
With
\[
  P(n)=\tfrac18\!\!\sum_{\substack{c^2+15d^2=4n\\ c\equiv d\,(2)}}\!\!(c^2-15d^2),
  \qquad
  R(n)=\!\!\sum_{\substack{c^2+15d^2=8n\\ c\equiv d\,(4)}}
        \Bigl(\tfrac{c^2-15d^2}{4}+\tfrac{cd}{2}\sqrt{-15}\Bigr)
  =R_r+R_i\sqrt{-15},
\]
the two newforms (\textsf{15.3.d.a}, \textsf{15.3.d.b}) are
$a_1=P+(\bar\omega/8)R$ and $a_2=P-(\bar\omega/8)R$; the $\sqrt{-15}$
parts cancel exactly, the coefficients are rational integers, and
$a_1(2)=+1$, $a_2(2)=-1$, $a_1(5)=+5$, $a_2(5)=-5$ pin Samart's
numbering $\#2=g_1$, $\#1=g_2$. Warning recorded for the machine: the
principal form of discriminant $-15$ is $x^2+xy+4y^2$, \emph{not}
$x^2+15y^2$; forms assembled from the latter orbit fail every Hecke
check.

\textbf{$K=\Qq(\sqrt{-21})$ (class number $4$, genus group
$(\Zz/2)^2$).} The ideal classes are $1,[\mathfrak p_2],
[\mathfrak p_3],[\mathfrak p_6]$ with $\mathfrak p_6=\mathfrak
p_2\mathfrak p_3$ and $\mathfrak p_2^2=(2)$, $\mathfrak p_3^2=(3)$,
$\mathfrak p_6^2=(6)$. The CM subspace of
$S_3(\Gamma_0(84),\chi_{-84})$ is four-dimensional with newforms
(\textsf{84.3.h.a}--\textsf{84.3.h.d}; the full newspace has
dimension $28$)
\[
  g(e_2,e_3)=P_0+e_2P_2+e_3P_3+e_2e_3P_6,\qquad e_2,e_3\in\{\pm1\},
\]
where
\begin{align*}
  P_0(n)&=\tfrac12{\sum_{a^2+21b^2=n}}'(a^2-21b^2),&
  P_3(n)&=\tfrac12{\sum_{3a^2+7b^2=n}}'(3a^2-7b^2),\\
  P_2(n)&=\tfrac14{\sum_{2x^2+2xy+11y^2=n}}'\bigl((2x+y)^2-21y^2\bigr),&
  P_6(n)&=\tfrac1{12}{\sum_{6x^2+6xy+5y^2=n}}'\bigl((6x+3y)^2-21y^2\bigr).
\end{align*}
All grössencharacter values are real: $\psi(\mathfrak p_2)=\pm2$,
$\psi(\mathfrak p_3)=\pm3$, $\psi(\mathfrak p_6)=\psi(\mathfrak
p_2)\psi(\mathfrak p_3)$; the convention $\psi(\mathfrak p_2)=+2$,
$\psi(\mathfrak p_3)=+3$ for $g(1,1)$ is pinned by
$a_{(1,1)}(2)=+2$, $a_{(1,1)}(3)=+3$. Samart's
$M_{84}^3=L'(g(-1,+1),0)$ and $M_{84}^4=L'(g(-1,-1),0)$.

\subsection{Ray-class projectors, conductor criteria, and
witnesses}\label{app:ray}

The twisted thetas use the projections $(1\pm\chi\circ N)/2$ with the
exact criteria
\begin{align*}
  \Qq(\sqrt{-2}):&\quad \chi_8\bigl(N(a+b\sqrt{-2})\bigr)=+1
    \iff b\ \text{even}\ (a\ \text{odd});\\
  \Qq(\ii):&\quad \chi_8\bigl(N(a+b\ii)\bigr)=+1
    \iff b\equiv0\pmod4\ (a\ \text{odd});\\
  \Qq(\sqrt{-6}):&\quad \chi_{-8}\bigl(N(a+b\sqrt{-6})\bigr)=+1
    \iff b\ \text{even}\ (a\ \text{odd}).
\end{align*}
Each is well defined modulo the conductor of
Table~\ref{tab:levels}: if $\alpha\equiv1\pmod{(2)}$ then $a$ is odd
and $b$ even, so $N(\alpha)\equiv a^2\equiv1\pmod8$ and
$\chi_8(N(\alpha))=\chi_{-8}(N(\alpha))=+1$; over $\Qq(\ii)$ the same
holds with $b\equiv0\pmod4$ for $\alpha\equiv1\pmod{(4)}$. The
conductors are exact, by one witness modulo each maximal proper
divisor. Over $\Qq(\sqrt{-2})$ and $\Qq(\sqrt{-6})$ the prime $2$
ramifies, $(2)=\mathfrak p_2^2$, and \emph{every} proper ideal
dividing $(2)$ already divides $\mathfrak p_2$; hence the single
witnesses
\[
  1+\sqrt{-2}\equiv1\ (\mathfrak p_2),\ N=3,\ \chi_8(3)=-1;\qquad
  1+\sqrt{-6}\equiv1\ (\mathfrak p_2),\ N=7,\ \chi_{-8}(7)=-1
\]
exclude every smaller modulus. Over $\Qq(\ii)$,
$(4)=(1+\ii)^4$, and every proper ideal dividing $(4)$ already
divides $(1+\ii)^3$; the witness
\[
  -1+2\ii\equiv1\ \bigl((1+\ii)^3\bigr)\ \text{since }
  (1+\ii)^3=(-2+2\ii),\qquad N(-1+2\ii)=5,\ \chi_8(5)=-1
\]
excludes every smaller modulus. (Over $\Qq(\sqrt{-3})$ the
conductor $(2)$ is exact: $2$ is inert and the theta identity of
\S\ref{app:newforms} has modulus $(2)$.)

\subsection{Genus characters and \texorpdfstring{$L_K$}{L(K)}-factorizations}\label{app:genus}

The genus-character identities (exact on the relevant norm sets) are
\[
  \chi_8(N)=\chi_{-3}(N)\ \text{on $\chi_{-24}$-norms};\qquad
  \chi_5(N)=\chi_{-3}(N)\ \text{on $\chi_{-15}$-norms};
\]
\[
  \chi_{-4}\chi_{84}=\chi_{21},\qquad
  \chi_{-3}\chi_{84}=\chi_{28},\qquad
  \chi_{-7}\chi_{84}=\chi_{12},
\]
and for class number $4$ the genus eigenvalues on the classes
$1,[\mathfrak p_2],[\mathfrak p_3],[\mathfrak p_6]$ are
\[
  (+,+,+),\quad (-,-,+),\quad (-,+,-),\quad (+,-,-)
  \qquad\text{in the basis }(\chi_{-4},\chi_{-3},\chi_{-7}).
\]
The factorizations
$L_K(\chi\circ N,s)=L(\chi,s)\,L(\chi\chi_K,s)$ use the pointwise
products $\chi_8\chi_{-8}=\chi_{-4}$, $\chi_8\chi_{-4}=\chi_{-8}$,
$\chi_5\chi_{-15}=\chi_{-3}$ (exact character arithmetic).

\subsection{The combination vectors}\label{app:comb}

Each Theorem's assembly is a \texttt{Fraction}-exact identity
$\mathrm{comb}=\sum c_i b_i$ in the displayed basis (the [S] tracks):
\[
\begin{array}{c|l|l}
\text{Thm.} & \text{basis }(b_i) & \text{comb vector }(c_i)\\
\hline
C & (L_8,\,L_8^{\mathrm{tw}},\,\zeta_K(2),\,z_V) &
    (7,\,8,\,\tfrac32,\,4)\\
D & (L_{16},\,L_{16}^{\mathrm{tw}},\,\zeta_K(2),\,z_{V_2}) &
    (5,\,8,\,\tfrac{27}{16},\,4)\\
E & (L_1,L_2,L_1^{\mathrm{tw}},L_2^{\mathrm{tw}},\zeta_K(2),L_g,B_1,B_2)
  & (\tfrac72,\,\tfrac32,\,4,\,4,\,\tfrac34,\,\tfrac74,\,2,\,2)\\
G & \text{as }D & (-20,\,32,\,-\tfrac{27}{4},\,16)\\
H & \text{as }E & (-14,\,-6,\,16,\,16,\,-3,\,-7,\,8,\,8)\\
I & (L_1,L_2,\zeta_K(2),L_g) & (6,\,8,\,\tfrac32,\,\tfrac{11}{2})
\end{array}
\]
with $L_g=L(\chi_8,2)L(\chi_{-3},2)$ (E, H),
$L_g=L(\chi_5,2)L(\chi_{-3},2)$ (I),
$B_1=L(\chi_{-8},2)L(\chi_3,2)$, $B_2=\mathrm{Cat}\,L(\chi_{24},2)$;
Theorem H's vector is $-4\times$Theorem E's with the character blocks
($L_i^{\mathrm{tw}},B_i$) sign-flipped. For Theorem I the split prime
$2$ contributes the exact Euler factors
$(1\mp\omega/8)(1\mp\bar\omega/8)=\tfrac{15}{16},\tfrac{19}{16}$.
For Theorems J, K the four class-lattice anchors are
\begin{align*}
  T(\OO_K)&=\tfrac12(z_K+A_4+A_3+A_7)+(L_1+L_2+L_3+L_4),\\
  T(\mathfrak p_2)&=\tfrac18(z_K-A_4-A_3+A_7)+\tfrac14(L_1+L_2-L_3-L_4),\\
  T(\mathfrak p_3)&=\tfrac1{18}(z_K-A_4+A_3-A_7)+\tfrac19(L_1-L_2+L_3-L_4),\\
  T(\mathfrak p_6)&=\tfrac1{72}(z_K+A_4-A_3-A_7)+\tfrac1{36}(L_1-L_2-L_3+L_4),
\end{align*}
with
\[
  z_K=\zeta(2)L(\chi_{-84},2),\quad
  A_4=L(\chi_{-4},2)L(\chi_{21},2),\quad
  A_3=L(\chi_{-3},2)L(\chi_{28},2),
\]
$A_7=L(\chi_{-7},2)L(\chi_{12},2)$, and
$L_i=L(g_i,3)$; the homotheties
$T(\Lambda_1)=16\,T(\mathfrak p_2)$, $T(\Lambda_2)=T(\OO_K)$ (K) and
$T(\Lambda_1)=1296\,T(\mathfrak p_6)$, $T(\Lambda_2)=81\,T(\mathfrak
p_3)$ (J) give
\[
  \mathrm{comb}_J=36(A_3-A_4)+72(L_3-L_4),\qquad
  \mathrm{comb}_K=4(A_3+A_4)+8(L_3+L_4),
\]
in which $z_K$ and $A_7$ cancel, and
$\EK_4=(5\sqrt{21}/(3\pi^3))\,\mathrm{comb}_J$,
$\EK_4=(5\sqrt{21}/\pi^3)\,\mathrm{comb}_K$.

\subsection{Closed Dirichlet values and their exact sources}\label{app:dirichlet}

The even-character values use
$L(\chi,2)=\pi^2\tau(\chi)\sum_a\chi(a)B_2(a/f)/f$
with exact integer character sums:
\[
\begin{array}{c|c|c}
\chi & \text{conductor} & \sum_a\chi(a)a^2\\
\hline
\chi_8 & 8 & 16\\
\chi_3 & 12 & 48\\
\chi_{24} & 24 & 288\\
\chi_5 & 5 & 4\\
\chi_{28} & 28 & 448\\
\chi_{21} & 21 & 168
\end{array}
\qquad
\begin{aligned}
  L(\chi_8,2)&=\pi^2\sqrt2/16,\\
  L(\chi_3,2)&=\pi^2\sqrt3/18,\\
  L(\chi_{24},2)&=\pi^2\sqrt6/24,\\
  L(\chi_5,2)&=4\pi^2\sqrt5/125,\\
  L(\chi_{28},2)&=2\sqrt7\pi^2/49,\\
  L(\chi_{21},2)&=8\sqrt{21}\pi^2/441.
\end{aligned}
\]
(Note the conductor-$12$ trap for $\chi_3$: its period table vanishes
at even positions.) Also
$L(\chi_{-28},2)=\frac{2\pi^2}{7\sqrt7}$ ($\sum_a\chi_{-28}(a)a^2=448$),
and $d_k=L'(\chi_{-k},-1)=\frac{k^{3/2}}{4\pi}L(\chi_{-k},2)$ by the
Dirichlet functional equation.

\subsection{Sturm data and newform identification}\label{app:sturm}

The eta forms are identified by Ligozat's criteria (weight, character,
cusp orders, divisor degree) and Sturm bounds
$\frac{3}{12}[\mathrm{SL}_2(\Zz):\Gamma_0(N)]$: bounds $3$ ($N=8$),
$6$ ($N=16$), $6$ ($N=12$), all below the exact-integer comparison
order $q^{60}$ (i.e.\ coefficients $a_0,\dots,a_{60}$ inclusive).
For the class-number-$2$ and $4$ fields the logical order is as
follows. Hecke's theorem on gr\"ossencharacters (quoted;
\cite{Miyake}) applied to the explicit characters of
\S\ref{app:newforms} produces bona fide newforms of weight $3$ and
of the displayed level and nebentypus; modularity is therefore a
theorem, not a script output. The exact script checks --- the Hecke
recursion $a(p^2)=a(p)^2-\chi(p)p^2$, multiplicativity, CM
vanishing at inert primes, all to exact order $q^{60}$ --- serve
only to \emph{identify} the assembled ideal-class theta sums with
those gr\"ossencharacter theta series, to fix the numbering, and
(via the Sturm bound, once both sides are known to lie in the same
finite-dimensional space) to promote the finite comparison to an
exact equality; on their own they would not imply the global Hecke
property. The LMFDB labels of \S\ref{app:newforms} record the
resulting identifications.

\subsection{Tail bounds for the interval tracks}\label{app:tails}

The [V] locks use only self-contained coefficient bounds:
$|a_n|\le(n+5)^5 2^{n/2}$ for the eta products (dominating every
factor $(1-q^{dn})^{r_d}$ by $(1-q^{dn})^{-|r_d|}$ and
$p(i)\le2^i$); $|a(n)|\le6n^3$ for $\Qq(\sqrt{-6})$ and
$|a(n)|\le30n^3$ for $\Qq(\sqrt{-15})$, $\Qq(\sqrt{-21})$ (explicit
representation-count boxes); geometric Mellin tails; closed
$\coth/\tanh$ lattice rows with the \texttt{tail\_row\_iv} bounds;
Euler--Maclaurin with Hurwitz tails for the Dirichlet values; closed
incomplete gamma at $s\in\{1,2,3\}$ and a squeezing continued fraction
for $E_1$.

\subsection{External inputs of the \texttt{cert0} chain}\label{app:cert0}

The only non-self-contained inputs of the $s_4$-value locks are
Shimura's CM theorem and the Shimura reciprocity law (integrality of
the CM values of the eta-quotient modular unit $s_4$, and completeness
of the Galois orbits). The locked minimal polynomials are:
$\Qq(\sqrt2)$ pairs (Theorems C, D and their conjugates F, G); the
$\Qq(\sqrt2,\sqrt3)$ quartic orbit (Theorem E, four CM/level points);
$X^2+192303X+1185921$ of discriminant $5\cdot85995^2$ (Theorem I);
$X^2+1787904X+84934656$ of discriminant $3\cdot1032192^2$ (Theorems J,
K). Individual-root locks use root separations from $305$ up to
$1.8\times10^6$ against lock widths $\le10^{-45}$.

% ============================================================
\section*{Statements and Declarations}
\paragraph{Funding.} The author did not receive support from any
organization for the submitted work.
\paragraph{Competing Interests.} The author has no relevant financial or
non-financial interests to disclose.
\paragraph{Data Availability.} All certification scripts, frozen
certificates, and run logs are available at
\url{https://github.com/huiminZheng-collab/samart-mahler}, archived at
\url{https://doi.org/10.5281/zenodo.21711884}.
\paragraph{Use of AI tools.} Documented in the Declaration on the use of
AI tools on the title page.

\end{document}